\documentclass[a4paper,10pt]{amsart}
\usepackage{amsmath,amssymb,mathtools,amsthm}
\usepackage{xcolor,graphicx}

\usepackage[colorlinks=true,urlcolor=blue, citecolor=red,linkcolor=blue,linktocpage,pdfpagelabels, bookmarksnumbered,bookmarksopen]{hyperref}
\usepackage[hyperpageref]{backref}
\usepackage{color}
\usepackage{verbatim}
\usepackage[makeroom]{cancel}

\usepackage[T1]{fontenc}

\usepackage[latin9]{inputenc}
\usepackage{geometry}
\newcommand{\vertiii}[1]{{\left\vert\kern-0.25ex\left\vert\kern-0.25ex\left\vert #1
		\right\vert\kern-0.25ex\right\vert\kern-0.25ex\right\vert}}

\newtheorem{theorem}{Theorem}[section]

\newtheorem{corollary}[theorem]{Corollary}
\newtheorem{lemma}[theorem]{Lemma}

\theoremstyle{definition}

\newtheorem{remark}[theorem]{Remark}

\numberwithin{equation}{section}

\newcommand{\nc}{\normalcolor}

\newcommand{\R}{\mathbb{R}}

\definecolor{darkblue}{rgb}{0.05, .05, .65}
\definecolor{darkgreen}{rgb}{0.1, .65, .1}
\definecolor{darkred}{rgb}{0.8,0,0}

\usepackage[]{cleveref}

\title{Characterisation of homogeneous fractional Sobolev spaces}

\author[Brasco]{Lorenzo Brasco}
\address[L.\ Brasco]{Dipartimento di Matematica e Informatica
\newline\indent
Universit\`a degli Studi di Ferrara
\newline\indent
Via Machiavelli 30, 44121 Ferrara, Italy}
\email{lorenzo.brasco@unife.it}

\author[Gómez-Castro]{David G\'omez-Castro}
\address[D.\ G\'omez-Castro]{
Mathematical Institute,
\newline\indent
University of Oxford,
\newline\indent
Radcliffe Observatory Quarter, Woodstock Road
\newline\indent
Oxford OX2 6GG, UK
\newline\indent
\textup{and}
\newline\indent
Instituto de Matem\'atica Interdisciplinar
\newline\indent
Universidad Complutense de Madrid
\newline\indent
Plaza de Ciencias 3, 28040 Madrid, Spain}
\email{gomezcastro@maths.ox.ac.uk}

\author[V\'azquez]{Juan Luis V\'azquez}
\address[J.\ L.\ V\'azquez]{Departamento de Matem\'aticas
\newline\indent
Universidad Aut\'onoma de Madrid,
\newline\indent
Ciudad Universitaria de Cantoblanco, 28049 Madrid, Spain}
\email{juanluis.vazquez@uam.es}


\begin{document}
	\maketitle

\begin{abstract}
 Our aim is to characterize the homogeneous fractional Sobolev-Slobodecki\u{\i} spaces $\mathcal{D}^{s,p} (\mathbb{R}^n)$ and their embeddings, for $s \in (0,1]$ and $p\ge 1$. They are defined as the completion of the set of smooth and compactly supported test functions with respect to the Gagliardo-Slobodecki\u{\i} seminorms.
 For $s\,p < n$ or $s = p = n = 1$ we show that $\mathcal{D}^{s,p}(\mathbb{R}^n)$ is isomorphic to a suitable function space, whereas for $s\,p \ge n$ it is isomorphic to a space of equivalence classes of functions, differing by an additive constant. As one of our main tools, we present a Morrey-Campanato inequality where the Gagliardo-Slobodecki\u{\i} seminorm controls from above a suitable
 Campanato seminorm.
\end{abstract}

\maketitle

\vspace{2cm}

\begin{center}
\begin{minipage}{10cm}
\small
\tableofcontents
\end{minipage}
\end{center}

\newpage

\section{Introduction}

\subsection{Fractional Sobolev spaces}
The aim of this paper is to shed light on an important topic in the theory  of  fractional Sobolev spaces. This family of spaces is conveniently presented e.g. in \cite{AdamsFourn2003, Leoni2009Sobolev, maz, Tartar2007}. It is common to define the fractional Sobolev  spaces  $W^{s,p} (\mathbb{R}^n)$ in the Sobolev-Slobodecki\u{\i} form. Thus, for $s \in (0,1)$
and $1\le p<+\infty$ we define the normalized
Gagliardo-Slobodecki\u{\i} seminorm by
\[
[u]_{W^{s,p} (\mathbb{R}^n)} =
\left(s\, (1-s) \,\iint_{\mathbb{R}^n\times\mathbb{R}^n} \frac{|u(x) - u(y)|^p}{|x-y|^{n+s\,p}}\, d x\, d y\right)^\frac{1}{p}.
\]
Then,
\begin{equation}
\label{wsp}
W^{s,p}(\mathbb{R}^n)\vcentcolon=\Big \{u\in L^p(\mathbb{R}^n)\, : \, [u]_{W^{s,p} (\mathbb{R}^n)}<+\infty\Big \},
\end{equation}
is a Banach space endowed with the non-homogeneous norm
\begin{equation*}
\|u \|_{W^{s,p} (\mathbb{R}^n)} = \|u \|_{L^p (\mathbb{R}^n)} + [u]_{W^{s,p} (\mathbb{R}^n)} .
\end{equation*}
A word about the convenience of the extra factor $s\,(1-s)$ is in order. Indeed, $[\,\cdot\,]_{W^{s,p} (\mathbb{R}^n)}$ can be thought as a real interpolation quantity, with parameter $s$, between the two quantities
\[
\int_{\mathbb{R}^n} |u|^p\,dx\qquad \mbox{ and }\qquad \int_{\mathbb{R}^n} |\nabla u|^p\,dx,
\]
see for example \cite{Brasco2019} or \cite{DTGVTaylor20}.
It is then natural to expect the following asymptotic behaviour
\[
\iint_{\mathbb{R}^n\times\mathbb{R}^n} \frac{|u(x) - u(y)|^p}{|x-y|^{n+s\,p}}\, d x\, d y\sim \frac{C}{s}\,\int_{\mathbb R^n} |u|^p\,dx,\qquad\mbox{ for } s\searrow 0,
\]
and
\[
\iint_{\mathbb{R}^n\times\mathbb{R}^n} \frac{|u(x) - u(y)|^p}{|x-y|^{n+s\,p}}\, d x\, dy\sim \frac{C}{1-s}\,\int_{\mathbb R^n} |\nabla u|^p\,dx,,\qquad\mbox{ for } s\nearrow 1,
\]
see \cite{MS} for the first result and \cite{bourgain} for the second one. For this reason, the factor $s\,(1-s)$ is incorporated in the definition of the seminorm and the limit cases $s=0$ and $s=1$
are defined accordingly by
\[
[u]_{W^{0,p} (\mathbb{R}^n)} = \|  u \|_{L^p (\mathbb{R}^n)}\qquad \mbox{ and }\qquad [u]_{W^{1,p} (\mathbb{R}^n)} = \| \nabla u \|_{L^p (\mathbb{R}^n)}.
\]
There are well-known embeddings of these spaces $W^{s,p} (\mathbb{R}^n)$ into $L^q(\mathbb{R}^n)$ for suitable $q\ge 1$ for which we refer to the classical monographs like \cite{AdamsFourn2003, Leoni2009Sobolev}.
\par
The particular case $s=(p-1)/p$ has a peculiar theoretical importance, since in this case
\[
W^{\frac{p-1}{p},p}(\mathbb{R}^n),
\]
can be identified with the {\it trace space} of functions in $W^{1,p}(\mathbb{H}^{n+1}_+)$, where
$\mathbb{H}^{n+1}_+=\mathbb{R}^n\times [0,+\infty)$.
More generally, it can be proved that $W^{s,p}(\mathbb{R}^n)$ coincides with the trace space of the weighted Sobolev space $\mathcal{W}_s^{1,p}(\mathbb{H}^{n+1}_+)$,
defined as
\[
\mathcal{W}_s^{1,p}(\mathbb{H}^{n+1}_+)=\Big\{u\in L^1_{\rm loc}(\mathbb{H}^{n+1}_+) \, : \, u\,y^{\frac{(p-1)-s\,p}{p}}\in L^p(\mathbb{H}^{n+1}_+),\  |\nabla u|\,y^{\frac{(p-1)-s\,p}{p}}\in L^p(\mathbb{H}^{n+1}_+)\Big\},
\]
where we have used the notation $(x,y)\in \mathbb{H}^{n+1}_+$, with $x\in \mathbb{R}^n$ and $y\in [0,+\infty)$. See \cite[Section 5]{Lio} for more details. The reader may also consult the recent paper \cite{MR}, containing some generalizations.

\subsection{Motivation for the homogeneous Sobolev spaces}

Before we present  the main results of this paper, we discuss some motivations for the study of {a particular} class of {fractional Sobolev} spaces.
Recently, there has been a surge of interest towards the study of nonlocal elliptic operators, that arise as first variations of Gagliardo-Slobodecki\u{\i} seminorms. The leading example is given by the {\it fractional Laplacian of order $s$} of a function $u$, indicated by the symbol $(-\Delta)^s u$, which in weak form reads as
\[
\langle(-\Delta)^s u, \varphi\rangle =
\iint_{\mathbb{R}^n\times\mathbb{R}^n} \frac{(u(x) - u(y))\,(\varphi(x)-\varphi(y))}{|x-y|^{n+2\,s}}\, d x\, d y,
\]
for all $\varphi\in C^\infty_c(\mathbb{R}^n)$, up to a possible normalization factor.
Observe that this is nothing but the first variation of the functional
\[
u\mapsto \frac{1}{2}\,[u]_{W^{s,2} (\mathbb{R}^n)}^2,
\]
up to the factor $s\,(1-s)$.
More generally, one could take a general exponent $1<p<+\infty$ and obtain accordingly the {\it fractional $p$-Laplacian of order $s$} of a function $u$, which we denote $(-\Delta_p)^s u$, and is defined in weak form by
\[
\langle(-\Delta_p)^s u, \varphi\rangle =
\iint_{\mathbb{R}^n\times\mathbb{R}^n} \frac{|u(x) - u(y)|^{p-2}\,(u(x)-u(y))\,(\varphi(x)-\varphi(y))}{|x-y|^{n+s\,p}}\, dx\, d y,
\]
for all $\varphi\in C^\infty_c(\mathbb{R}^n)$, up to a multiplicative factor.
This is a nonlocal and nonlinear operator which has been extensively studied in recent years, see \cite{Vaz16, Vaz20} and references therein.
\normalcolor
In order to motivate the studies performed in this paper, let us consider
the quasilinear nonlocal elliptic problem
\[
(-\Delta_p)^s u=f,\qquad \mbox{ in }\mathbb{R}^n,
\]
under suitable assumptions on the source term $f$. In order to prove existence of a weak solution, it would be natural
to use the Direct Method in the Calculus of Variations. This would lead to the problem of minimizing the energy functional
\begin{equation}
\label{functional}
u\mapsto \frac{1}{p}\,[u]_{W^{s,p} (\mathbb{R}^n)}^p-\int_{\mathbb{R}^n} f\,u\,dx,
\end{equation}
which is naturally associated to our equation. However, it is not clear the functional space where this minimization problem should be posed. For example, one could try to pose the problem in the  previously introduced space $W^{s,p}(\mathbb{R}^n)$, but it is easily seen that this does not fit at all. Indeed, the functional \eqref{functional} {\it is not weakly coercive} on this space, unless we are in the trivial situation $f\equiv 0$. Of course, the problem is that the functional \eqref{functional} can not permit to infer any control on the $L^p$ norm of minimizing sequences. This in turn is related to the fact that the Poincar\'e inequality
\[
c\,\int_{ \mathbb{R}^n }|u|^p\,dx  \le [u]^p_{W^{s,p} ( \mathbb{R}^n )}
\]
{\it fails to be true on the whole $\mathbb{R}^n$} for any $c>0$. This can be easily seen by using the invariance of $\mathbb{R}^n$ with respect to scalings $x\mapsto \lambda\,x$ and a simple dimensional analysis of the two norms.

\medskip

It turns out that the natural spaces to work with are the {\it homogeneous Sobolev spaces} $\mathcal{D}^{s,p} (\mathbb{R}^n)$. They are defined by
\[
\mathcal{D}^{s,p} (\mathbb{R}^n)  \vcentcolon= ``\text{completion of } C_c^\infty (\mathbb{R}^n) \text{ with respect to } [\,\cdot\,]_{W^{s,p}(\R^n)} ".
\]
We recall that $W^{s,p}(\mathbb{R}^n)$ and $\mathcal{D}^{s,p}(\mathbb{R}^n)$ are particular instances of  the huge family of {\it Besov spaces}. In this respect, we stress that a mention of homogeneous Besov spaces can be found for example in \cite[Remark 7.68]{AdamsFourn2003}, \cite[Chapter 6, Section 3]{BL}, \cite[Chapter 10, Section 1]{maz} and \cite[Chapter 3, Section 4]{Tri}, among others.

The notation $\mathcal{D}^{s,p}(\mathbb{R}^n)$ adopted here is reminiscent of the historical one, introduced by Deny and Lions in their paper \cite{DL}. This reference has been among the first papers to study homogeneous spaces, obtained by completion of  $C^\infty_c(\mathbb{R}^n)$.

Another frequently encountered notation for homogeneous Sobolev spaces is $\dot W^{s,p}(\mathbb{R}^n)$, see for example Petree's paper \cite{Pe}. However, usually this notation is used for spaces of functions identified modulo constants for $s\in(0,1]$. We will adopt the same convention in this paper.

It is our aim to examining the space $\mathcal{D}^{s,p}(\mathbb{R}^n)$ more closely and connect it with spaces of the type $\dot W^{s,p}(\mathbb{R}^n)$, as we will explain in a moment.

\subsection{Completions} Before presenting the main results of the paper, let us briefly recall some basic facts about the completion process. We start from the seminorm $[\,\cdot\,]_{W^{s,p} (\mathbb{R}^n)} $ for $0<s \le 1$ and $1<p<\infty$. It is not difficult to see that this turns out to be a norm on the space $C_c^\infty (\mathbb{R}^n)$.
\par
However, the normed space
\[
\left(C_c^\infty(\mathbb{R}^n), [\,\cdot\,]_{W^{s,p}(\mathbb{R}^n)}\right),
\]
is not complete.
By definition, its completion is the quotient space of the set of sequences $(u_m)_{m\in\mathbb{N}} \subset C_c^\infty (\mathbb{R}^n)$
which are Cauchy for the norm $[\,\cdot\,]_{W^{s,p} (\mathbb{R}^n)}$, under the expected equivalence relation
\[
(u_m)_{m\in\mathbb{N}} \sim_{s,p} (v_m)_{m\in\mathbb{N}}\qquad \mbox{  if  }\qquad \lim_{m\to\infty}[u_m - v_m]_{W^{s,p} (\mathbb{R}^n)}=0.
\]
For each equivalence class $U = \{(u_m)_{m\in\mathbb{N}}\}_{s,p}\in\mathcal{D}^{s,p}(\mathbb{R}^n)$, we define its norm in terms of a representative as
\[
\left\| U \right\|_{\mathcal{D}^{s,p} (\mathbb{R}^n)} \vcentcolon= \lim_{m\to\infty} [u_m]_{W^{s,p} (\mathbb{R}^n)}.
\]
It is easily seen that such a definition is independent of the  chosen representative.
With this construction,
\[
\left(\mathcal{D}^{s,p} (\mathbb{R}^n), \|\,\cdot\,\| _{\mathcal{D}^{s,p} (\mathbb{R}^n)} \right)
\]
is a Banach space.
Note that all functions $u \in C_c^\infty (\mathbb{R}^n)$ can be naturally embedded into $\mathcal{D}^{s,p} (\mathbb{R}^n)$ by the constant sequence $u_m = u$. By definition, this representation of $C_c^\infty (\mathbb{R}^n)$ is dense in $\mathcal{D}^{s,p} (\mathbb{R}^n)$.

\smallskip
In the case $s=1$, it is well-known that $\mathcal{D}^{1,p}(\mathbb{R}^n)$ is a subspace of the space of distributions $\mathcal{D}'(\mathbb{R}^n)$ if and only if $1\le p<n$. In this case, this can be identified with the functional space
\[
\Big\{u\in L^{p^*}(\mathbb{R}^n)\, :\, \nabla u\in L^p(\mathbb{R}^n;\mathbb{R}^n)\Big\},\qquad \mbox{ where } p^*=\frac{n\,p}{n-p},
\]
thanks to the celebrated Sobolev inequality. The case $p=2$ is contained in Deny and Lions, \cite[Th\'eor\`eme 4.4 and Remark 4.1]{DL}. The general case can be found for example in \cite[Chapter 15]{maz}.
On the contrary, the case $p\ge n$ is much more delicate, since in this case $\mathcal{D}^{1,p}(\mathbb{R}^n)$ is not even a subspace of  $\mathcal{D}'(\mathbb{R}^n)$.  A concrete characterization of $\mathcal{D}^{1,p}(\mathbb{R}^n)$ as a space of equivalence classes of functions modulo constants seems to belong to the folklore on the subject, though we have not been able to find a proper reference in the literature. Our presentation will cover this case, as well.

\subsection{Main results: the three ranges}
The main question we address in these notes is the characterization of $\mathcal{D}^{s,p}(\mathbb{R}^n)$ and the study of some of its embeddings into suitable sets of functions. This will be possible for some exponents, while for other exponents the embedding occurs into a space of equivalence classes of functions modulo constants, as we are now going to explain.

More precisely, in order to answer to these questions, we will need to distinguish three cases, according to the different behaviours of
\[
u\mapsto [u]_{W^{s,p}(\mathbb{R}^n)},
\]
with respect to scalings of the form $x\mapsto \lambda\,x$, with $\lambda>0$.
By a simple change in variable, we have that
\begin{equation}
\label{eq:scaling}
[u_\lambda]_{W^{s,p}(\mathbb{R}^n)}= \lambda^{s - \frac n p } [u]_{W^{s,p}(\mathbb{R}^n)},\qquad \mbox{ for every } \lambda>0, \mbox{ where } u_\lambda(x)=u(\lambda\,x),
\end{equation}
This shows that the relation between $s\,p$ and $n$ provides significantly different results. Consequently, there are three different situations:

\subsubsection*{\sc Subconformal case $s\,p<n$} The natural inclusion of $\mathcal{D}^{s,p}(\mathbb{R}^n)$ into a Lebesgue space of functions succeeds. More precisely, the completion can be identified with a functional space, i.\,e., we have
\[
\mathcal{D}^{s,p}(\mathbb{R}^n)\simeq \dot W^{s,p}(\mathbb{R}^n)\vcentcolon=\Big\{  u \in L^{p^\star_s} ( \mathbb{R}^n )\, :\,  [ u ]_{W^{s,p} (\mathbb{R}^n)} < +\infty  \Big\}, \qquad \mbox{ where } \, p^\star_s = \frac{n\,p}{n-s\,p},	
\]	
see Theorem \ref{thm:sp < n} below. The main tool here is the fractional Sobolev inequality.
	This mimics in some sense what happens for $W^{s,p} (\mathbb{R}^n)$, that can be characterized by \eqref{wsp}
as the completion of $C_c^\infty (\mathbb{R}^n)$ with  respect to the  $W^{s,p} (\mathbb{R}^n)$ norm.

\subsubsection*{\sc Superconformal case $s\,p > n$} Here, the Sobolev inequality is not available and we have to replace it by {\it Morrey's inequality}, see equation \eqref{eq:Morrey} below. However, unlike the case $s\,p<n$, the elements in $\mathcal{D}^{s,p} (\mathbb{R}^n)$ {\it can not} be uniquely represented by functions.
Indeed, when $s\,p > n$, there exist sequences $(\varphi_m)_{m\in\mathbb{N}}$ such that
\begin{equation}
\label{successionem}
[ \varphi_m ] _{W^{s,p} (\mathbb{R}^n)} \to 0
\qquad \text{ and } \qquad
\varphi_m \to 1 \text{  uniformly over compact sets},
\end{equation}
as $m$ goes to $\infty$. These sequences are known as \emph{null-sequences}.
Hence, any sequence $(u_m)_{m\in\mathbb{N}}\subset C_c^\infty (\mathbb{R}^n)$ which is Cauchy in the norm $[\,\cdot\,]_{W^{s,p} (\mathbb{R}^n)}$ is equivalent to the sequence
 \[
 v_m = u_m + C\, \varphi_m,
 \]
for any constant $C\in\mathbb{R}$.
\nc
Observe that this implies in particular that now {\it all constant functions are equivalent to the null one in} $\mathcal{D}^{s,p}  (\mathbb{R}^n)$.

Furthermore, one can show that functions that are approximated by equivalent Cauchy sequences actually coincide up to a constant. This allows to show that $\mathcal{D}^{s,p}(\mathbb{R}^n)$ can be identified with a space of equivalence classes of H\"older continuous functions differing by an additive constant, i.\,e. we have
\[
\mathcal{D}^{s,p}(\mathbb{R}^n)\simeq \dot W^{s,p} (\mathbb{R}^n) \vcentcolon=  \frac{\Big\{  u \in C^{0,s-\frac{n}{p}}(\mathbb{R}^n)\, :\, [ u ]_{W^{s,p} (\mathbb{R}^n)} < +\infty  \Big\}}{\sim_C},
\]
where for $0<\alpha\le 1$
\[
C^{0,\alpha}(\mathbb{R}^n)=\left\{u:\mathbb{R}^n\to\mathbb{R}\,:\, \sup_{x\not=y}\frac{|u(x)-u(y)|}{|x-y|^\alpha} < +\infty\right\},
\]
and $\sim_C$ is the equivalence relation
\begin{equation}
\label{eq:equivalence up to constant}
u\sim_C v \qquad \Longleftrightarrow\qquad u-v \mbox{ is constant}.
\end{equation}
We refer to Theorem \ref{thm:sp > n} below, for complete details.

\subsubsection*{\sc Conformal case $s\,p = n$.} This is the most delicate case. Whenever $s<n$ (i.\,e., unless $s=n=1$), it is still possible to prove existence of a sequence $(\varphi_m)_{m\in\mathbb{N}}$ such that properties \eqref{successionem} hold. However, the construction of such a sequence is now more involved. Since the seminorm is scale invariant in this case, such a construction can not be just based on scalings. As in the local case $s=1$, one has to consider a suitable sequence of truncated and rescaled logarithms (see Lemma \ref{lem:confo} below). This is reminiscent of the optimal sequence for the Moser-Trudinger inequality, see for example \cite[Section 5]{Parini2019} for the fractional case.

This permits to show that also in the case $s\,p=n$ (provided $s < n$),
we can approximate the constant functions by functions in the null class. Hence, this case behaves like $s\,p > n$
and $\mathcal{D}^{s,n/s}(\mathbb{R}^n)$ can be identified with a space of equivalence classes of $BMO$ functions differing by an additive constant, i.e. we have
\[
\mathcal{D}^{s,\frac{n}{s}}(\mathbb{R}^n)\simeq \dot W^{s,\frac{n}{s}} (\mathbb{R}^n) \vcentcolon= \frac{\Big\{  u \in BMO(\mathbb{R}^n)\, :\, [ u ]_{W^{s,\frac{n}{s}} (\mathbb{R}^n)} < +\infty  \Big\}}{\sim_C},
\]
see Theorem \ref{thm:sp = n} below.

  Still, there will be room for a small surprise. Indeed, we will show that the limiting case $s=p=n=1$, which still falls in the conformal regime, behaves like $s\,p < n$. In other words, the homogeneous Sobolev space $\mathcal{D}^{1,1}(\mathbb{R})$ is actually a function space and we have
\[
\mathcal{D}^{1,1}(\mathbb{R})
\simeq \dot W^{1,1} (\mathbb{R}^n) \vcentcolon=
 \left\{  u  \in C_0 (\R)  : \int_{\mathbb{R}} |u'|\,dx < +\infty \right \},
\]
where $C_0(\mathbb{R})$  is the space  of continuous functions vanishing at infinity, see Theorem \ref{teo:spn1} below.
In this way we complete the characterization of the spaces.

\begin{remark}
After completing this work, we became aware of the interesting recent paper \cite{MPS}, dealing with the same issue here addressed, but for a different scale of fractional Sobolev spaces. Namely, the authors of \cite{MPS} deal with the so-called {\it Bessel potential spaces} (sometimes also called {\it Liouville spaces}), which are defined in terms of the Fourier transform. In \cite[Theorem 2]{MPS}, they give a concrete realization of the homogeneous version of these spaces.
\par
Our results only partially superpose with those of \cite{MPS} and, in any case, the proofs are different.
Indeed, the results of \cite{MPS} are based on Harmonic Analysis techniques and contain ours only for the case $p=2$ and $0<s<1$, and for $1<p<\infty$ and $s=1$. 
\end{remark}

\subsection{Plan of the paper}
We start with Section \ref{sec:2}, where some basic facts about $BMO$ and Campanato spaces are recalled. These tools are particularly useful to handle the cases $s\,p\ge n$.
In this part, an important result is Theorem \ref{teo:campanatoholder}, which relates the Gagliardo-Slobodecki\u{\i} seminorm and a Campanato seminorm.
This yields as corollaries several important inequalities: a fractional Poincaré--Wirtinger inequality (see Corollary \ref{coro:poincare}) and the fractional Morrey inequality for $s\, p > n$ (see Corollary \ref{coro:morrey}).
\par
We devote  \Cref{sec:3}, \Cref{sec:4} and \Cref{sec:5}  to prove the characterisation of $\mathcal D^{s,p}$  in the cases $s\,p<n$,  $s\,p>n$, and $s\,p=n$, respectively. We introduce in each case suitable structural lemmas.
\par
We conclude the paper with two appendices on approximation lemmas, which will allow to show that the elements in $\dot W^{s,p} (\mathbb R^n)$ can be approximated by functions in $C_c^\infty (\mathbb{R}^n)$. The aim is to approximate a function $u$ in the Gagliardo--Slobodecki\u{\i} seminorm by sequences of the type $(u \ast \rho_m)\, \eta_m$, where $\rho_m$ are standard mollifiers and $\eta_m$ are cut-off functions.
In Appendix \ref{sec:convolution} we study the convolution, while in Appendix \ref{sec:6} we prove several
{\it truncation lemmas}, which allow to estimate the effect of multiplying by cut-off functions.

\section{Preliminaries}
\label{sec:2}

\subsection{BMO and Campanato spaces}
At first, we need to recall definitions and some basic facts about {\it bounded mean oscillation functions} and {\it Campanato spaces}. As already announced, this will be particularly useful to deal with the cases $s\,p>n$ and $s\,p=n$. We will indicate by $B_r(x_0)$ the $n$-dimensional open ball with center $x_0\in\mathbb{R}^n$ and radius $r>0$. The symbol $\omega_n$ will stand for the measure of $B_1(0)$.
\medskip

A function of bounded mean oscillation is a locally integrable function $u$ such that the supremum of its mean oscillations is finite. More precisely, for every $u\in L^1_{\rm loc}(\mathbb{R}^n)$ we define
$$
[u]_{BMO(\mathbb R^n)} = \sup_{ x_0 \in \mathbb R^n, \varrho > 0}
\frac1{|B_\varrho (x_0)|}\int_{B_\varrho (x_0) }|u(x)-u_{x_0, \varrho} |\,dx,
$$
where
\[
u_{x_0,\varrho}=\frac{1}{|B_\varrho(x_0)|}\,\int_{B_\varrho(x_0)} u\,dx.
\]
Then we define the space of functions with bounded mean oscillation as
\[
BMO(\mathbb{R}^n)=\Big\{u\in L^1_{\rm loc}(\mathbb{R}^n)\, :\, [u]_{BMO(\mathbb{R}^n)}<+\infty\Big\}.
\]
 This space was introduced by John and Nirenberg in \cite{JN}.
The $BMO$ space is a borderline space, which plays a key  role  in different areas of Mathematical Analysis, as a natural replacement of $L^\infty(\mathbb{R}^n)$ in a large number of results, for instance in interpolation. 	Fefferman and Stein characterized this space as the dual of the Hardy space $\mathcal{H}^1$, see \cite[Theorem 1]{Feff71} and \cite[Theorem 2]{FeSt}. Another important appearance of this space is in Elliptic Regularity Theory: indeed, the logarithm of a positive local solution to an elliptic partial differential equation is a locally $BMO$ function. This observation is a crucial step in the classical proof by Moser of Harnack's inequality, see \cite{Mo}.

It turns out that the $BMO$ space can be seen as a particular instance of the larger family of \ {\sl Campanato spaces}, see \cite{Cam63}.	For $1\le p<+\infty$ and $0\le \lambda\le n+p$, for every $u\in L^p_{\rm loc}(\mathbb{R}^n)$ we define  the seminorm
\[
[u]_{\mathcal{L}^{p,\lambda}(\mathbb{R}^n)}=\left(\sup_{x_0\in\mathbb{R}^n,\varrho>0} \varrho^{-\lambda}\,\int_{B_\varrho(x_0)} |u-u_{x_0,\varrho}|^p\,dx\right)^\frac{1}{p}=
\sup_{x_0\in\mathbb{R}^n,\varrho>0} \varrho^{-\lambda/p}\left(\,\int_{B_\varrho(x_0)} |u-u_{x_0,\varrho}|^p\,dx\right)
^\frac{1}{p}.
\]
Accordingly, we introduce the Campanato space
\[
\mathcal{L}^{p,\lambda}(\mathbb{R}^n)=\Big\{u\in L^p_{\rm loc}(\mathbb{R}^n)\, :\, [u]_{\mathcal{L}^{p,\lambda}(\mathbb{R}^n)}<+\infty\Big\}.
\]
Notice that since $|B_\varrho(x_0)| = \omega_n\, \varrho^n$ we have
	\begin{equation*}
		[u] _{ BMO(\mathbb R^n)}  = \frac{1}{\omega_n}\, [u]_{\mathcal L^{1,n} (\mathbb R^n)}, \qquad
BMO(\mathbb R^n)  =\mathcal L^{1,n} (\mathbb R^n).
	\end{equation*}
	We recall that for $n<\lambda\le n+p$, we have
\begin{equation}
\label{campanatoholder}
\frac{1}{C}\,[u]_{C^{0,\alpha}(\mathbb{R}^n)}\le [u]_{\mathcal{L}^{p,\lambda}(\mathbb{R}^n)}\le C\,[u]_{C^{0,\alpha}(\mathbb{R}^n)},\qquad \mbox{ with } \alpha=\frac{\lambda-n}{p},
\end{equation}
see \cite[Chapter 2, Section 3]{Gi}. The constant $C=C(\lambda,n,p)>0$ blows-up as $\lambda\searrow n$.

\smallskip

\noindent We must point out that both $[\,\cdot\,]_{BMO(\mathbb{R}^n)}$ and $[\,\cdot\,]_{\mathcal{L}^{p,\lambda}(\mathbb{R}^n)}$
are only seminorms on their relevant spaces, as they do not detect constants, i.e. the seminorm of constant functions is zero. In order to get a normed space (actually, a Banach space), we need to consider their homogeneous version defined as quotient spaces
\begin{equation*}
\dot{BMO}(\mathbb{R}^n)=\frac{BMO (\mathbb R^n)}{\sim_C}\qquad \mbox{ and }\qquad \dot {\mathcal{L}}^{p,\lambda}(\mathbb{R}^n)=\frac{\mathcal{L}^{p,\lambda}(\mathbb{R}^n)}{\sim_C},
\end{equation*}
with the equivalence relation
\begin{equation*}
\label{eq:equivalence relation up to constant}
u\sim_C v \qquad \Longleftrightarrow\qquad u-v \mbox{ is constant almost everywhere}.
\end{equation*}
We will denote by $\{ u\}_{C}$ the class of $u$ with respect to this relation.

An interesting result, which can be found for example in \cite[Proposition 2.5 and Corollary 2.3]{Gi}, says that  	when $\lambda = n$ all the Campanato spaces are isomorphic and we have
\begin{equation}\label{BMOL}
	\dot{BMO}(\mathbb{R}^n) \simeq  \dot {\mathcal{L}}^{p,n}(\mathbb{R}^n), \qquad \mbox{ for every } 1\le p<+\infty.
\end{equation}
On the contrary, for $\lambda\not=n$	 the spaces $\mathcal L^{p,\lambda}(\mathbb{R}^n)$ and $\mathcal{L}^{q,\lambda}(\mathbb{R}^n)$ do not coincide, for $p\not =q$.

	\begin{remark} As pointed out in \cite[Chapter 4, \S 1.1.1]{Stein}, the definition of the $BMO$ space can be equivalently given by taking hypercubes instead of balls. The same remark applies to Campanato spaces.
	\end{remark}


\subsection{Weighted integrability for some Campanato spaces}

Functions belonging to $\mathcal{L}^{p,n}(\mathbb{R}^n)$ enjoy a suitable weighted global integrability condition. More precisely, we have the following
\begin{lemma}
	\label{lm:FeSt}
	Let $1\le p<+\infty$ and $R>0$. For every $u\in \mathcal{L}^{p,n}(\mathbb{R}^n)$ such that
	\[
	\int_{B_R(0)} u\,dx=0,
	\]
	we have
	\begin{equation}
	\label{fantasy}
	\int_{\mathbb{R}^n} \frac{|u|^p}{R^n+|x|^n\,\left|\log\dfrac{|x|}{R}\right|^{p+2}}\,dx\le C\,[u]^p_{\mathcal{L}^{p,n}(\mathbb{R}^n)},
	\end{equation}
	for some constant $C=C(n,p)>0$.
\end{lemma}

\begin{proof}
The proof is an adaptation of that of \cite[Equation (1.2)]{FeSt}. The final outcome is slightly better. We observe that it is sufficient to prove \eqref{fantasy} for $R=1$. The general case then follows by a standard scaling argument.

\noindent We first fix some shortcut notation, for the sake of simplicity. For every $k\in\mathbb{N}$, we set
	\[
	u_k=\frac{1}{|B_{2^k}(0)|}\,\int_{B_{2^k}(0)} u\,dx,
	\]
	and observe that $u_0=0$, by assumption.
	We start by estimating the difference $|u_{k+1}-u_{k}|$. We have
	\[
	\begin{split}
	\left|\int_{B_{2^{k}}(0)} [u(x)-u_{k+1}] \,dx\right|&\le \int_{B_{2^{k}}(0)} |u(x)-u_{k+1}| \,dx\\
	&\le |B_{2^{k}}(0)|^\frac{p-1}{p}\,\left(\int_{B_{2^{k}}(0)} |u(x)-u_{k+1}|^p \,dx\right)^\frac{1}{p}\\
	&\le |B_{2^{k}}(0)|^\frac{p-1}{p}\,\left(\int_{B_{2^{k+1}}(0)} |u(x)-u_{k+1}|^p \,dx\right)^\frac{1}{p}\\
	&\le C\,|B_{2^{k}}(0)|^\frac{p-1}{p}\,|B_{2^{k+1}}(0)|^\frac{1}{p}\,[u]_{\mathcal{L}^{p,n}(\mathbb{R}^n)}.
	\end{split}
	\]
	By observing that
	\[
	\left(\frac{|B_{2^{k+1}}(0)|}{|B_{2^k}(0)|}\right)^\frac{1}{p}=2^\frac{N}{p},
	\]
	we can divide both sides by $|B_{2^{k}}(0)|$ and get
	\[
	|u_k-u_{k+1}|\le C\,[u]_{\mathcal{L}^{p,n}(\mathbb{R}^n)}.
	\]
	By using this estimate and the triangle inequality, we then get for every $j\in\mathbb{N}\setminus\{0\}$
	\[
	|u_j|=|u_j-u_0|\le \sum_{k=0}^{j-1} |u_{k+1}-u_k|\le C\,j\,[u]_{\mathcal{L}^{p,n}(\mathbb{R}^n)}.
	\]
	We now get
	\[
	\begin{split}
	\int_{B_{2^j}(0)} |u|^p\,dx=\int_{B_{2^j}(0)} |u-u_0|^p\,dx&\le 2^{p-1}\,\int_{B_{2^j}(0)} |u-u_j|^p\,dx+2^{p-1}\,|u_j|^p\,|B_{2^j}(0)|\\
	&\le 2^{p-1}\,\int_{B_{2^j}(0)} |u-u_j|^p\,dx+2^{p-1}\,C^p\,j^p\,|B_{2^j}(0)|\,[u]_{\mathcal{L}^{p,n}(\mathbb{R}^n)}^p\\
	&\le C\,|B_{2^j}(0)|\,(1+j^p)\,[u]_{\mathcal{L}^{p,n}(\mathbb{R}^n)}^p.
	\end{split}
	\]
	In particular, this implies for every $j\in\mathbb{N}\setminus\{0\}$
	\[
	\frac{1}{|B_{2^j}(0)|\,(1+j^p)}\,\int_{B_{2^j}(0)\setminus B_{2^{j-1}}(0)} |u|^p\,dx\le C\,[u]_{\mathcal{L}^{p,n}(\mathbb{R}^n)}^p.
	\]
	We now observe that for every $x\in B_{2^j}(0)\setminus B_{2^{j-1}}(0)$
	\[
	|B_{2^j}(0)|\,(1+j^p)=\omega_n\,2^{j\,n}\,(1+j^p)\le \omega_n\,\left(2\,|x|\right)^n\,\Big(1+\left(1+\log_2 |x|\right)^p\Big).
	\]
	If we use this estimate in the previous inequality, we get
	\[
	\int_{B_{2^j}(0)\setminus B_{2^{j-1}}(0)} \frac{|u|^p}{1+|x|^n\,(\log |x|)^p}\,dx\le C\,[u]_{\mathcal{L}^{p,n}(\mathbb{R}^n)}^p.
	\]	
We further divide both sides by $j^2$, so to get
\[
\frac{1}{j^2}\,\int_{B_{2^j}(0)\setminus B_{2^{j-1}}(0)} \frac{|u|^p}{1+|x|^n\,(\log |x|)^{p}}\,dx\le \frac{C}{j^2}\,[u]_{\mathcal{L}^{p,n}(\mathbb{R}^n)}^p.
\]
On the left-hand side, we use that for every $x\in B_{2^j}(0)\setminus B_{2^{j-1}}(0)$ it holds
\[
j^2\le \left(1+\log_2 |x|\right)^2\le C\,\Big(1+(\log |x|)^2\Big).
\]
This in turn implies that
\[
j^2\,\Big(1+|x|^n\,(\log |x|)^p\Big)\le C\,\Big(1+|x|^n\,(\log |x|)^{p+2}\Big),
\]
possibly for a different constant $C>0$.
If we now sum over $j\ge 1$
we get
\begin{equation}
\label{1}
\int_{\mathbb{R}^n\setminus B_1(0)} \frac{|u|^p}{1+|x|^n\,(\log |x|)^{p+2}}\,dx\le C\,[u]_{\mathcal{L}^{p,n}(\mathbb{R}^n)}^p.
\end{equation}	
We are only left with observing that we have (recall that $u$ has average $0$ in $B_1(0)$)
\[
\int_{B_1(0)} |u|^p\,dx\le C\,|B_1(0)|\,[u]_{\mathcal{L}^{p,n}(\mathbb{R}^n)}^p,
\]
and
\[
1+|x|^n\,\Big|\log |x|\Big|^{p+2}\ge \frac{1}{C},\qquad \mbox{ for } x\in B_1(0).
\]
Thus we get
\begin{equation}
\label{2}
\int_{B_1(0)} \frac{|u|^p}{1+|x|^n\,\Big|\log |x|\Big|^{p+2}}\,dx\le C\,[u]_{\mathcal{L}^{p,n}(\mathbb{R}^n)}^p,
\end{equation}
as well. By summing up \eqref{1} and \eqref{2}, we get \eqref{fantasy} for $R=1$, as desired.
\end{proof}

\begin{remark}Due to equation \eqref{BMOL}, the previous weighted estimate applies to $BMO$, as well.\end{remark}

\subsection{A Morrey-Campanato--type inequality and applications}

We now prove an inequality relating the Gagliardo-Slobodecki\u{\i} and Campanato seminorms. This will give us, as corollaries, fractional versions of the Poincaré--Wirtinger and Morrey inequalities.
\begin{theorem}
\label{teo:campanatoholder}
Let $s\in(0,1)$ and $1\le p<+\infty$.
Then for every $u\in C^\infty_c(\mathbb{R}^n)$ we have
\[
[u]_{\mathcal{L}^{p,sp}(\mathbb{R}^n)}\le C\,[u]_{W^{s,p}(\mathbb{R}^n)},
\]
for a constant $C=C(n,p)>0$.
\end{theorem}
\begin{proof}
We fix $x_0\in\mathbb{R}^n$ and $\varrho>0$, then for every $u,v\in C^\infty_c(\mathbb{R}^n)$ by Minkowski inequality we get
\[
\begin{split}
\left(\varrho^{-s\,p}\,\int_{B_\varrho(x_0)} |u-u_{x_0,\varrho}|^p\,dx\right)^\frac{1}{p}&\le \left(\varrho^{-s\,p}\,\int_{B_\varrho(x_0)} |u-v|^p\,dx\right)^\frac{1}{p}+\left(\varrho^{-s\,p}\,\int_{B_\varrho(x_0)} |v-v_{x_0,\varrho}|^p\,dx\right)^\frac{1}{p}\\
&+\left(\varrho^{-s\,p}\,\int_{B_\varrho(x_0)} |v_{x_0,\varrho}-u_{x_0,\varrho}|^p\,dx\right)^\frac{1}{p}\\
&\le2\,\left(\varrho^{-s\,p}\,\int_{\mathbb{R}^n} |u-v|^p\,dx\right)^\frac{1}{p}+\left(\varrho^{-s\,p}\,\int_{B_\varrho(x_0)} |v-v_{x_0,\varrho}|^p\,dx\right)^\frac{1}{p}.
\end{split}
\]
In the second estimate, we used Jensen's inequality.
We now apply the standard Poincar\'e-Wirtinger inequality (see for example \cite[Theorem 3.17]{Gi}) in order to control the last term
\[
\int_{B_\varrho(x_0)} |v-v_{x_0,\varrho}|^p\,dx\le C\,\varrho^p\,\int_{B_\varrho(x_0)} |\nabla v|^p\,dx\le C\,\varrho^p\,[v]^p_{W^{1,p}(\mathbb{R}^n)},
\]
for a constant $C=C(n,p)>0$.
Then the last two displays imply the estimate
\[
\left(\varrho^{-s\,p}\,\int_{B_\varrho(x_0)} |u-u_{x_0,\varrho}|^p\,dx\right)^\frac{1}{p}\le C\,\frac{\|u-v\|_{L^p(\mathbb{R}^n)}+\varrho\,[v]_{W^{1,p}(\mathbb{R}^n)}}{\varrho^s}.
\]
This estimate is valid for every $v\in C^\infty_c(\mathbb{R}^n)$. Thus, if we define the $K-$functional
\[
K(t,u):=\inf_{v\in C^\infty_c(\mathbb{R}^n)}\Big(\|u-v\|_{L^p(\mathbb{R}^n)}+t\,[v]_{W^{1,p}(\mathbb{R}^n)}\Big),
\]
by taking the infimum over $v$ in the last estimate, we get
\[
\left(\varrho^{-s\,p}\,\int_{B_\varrho(x_0)} |u-u_{x_0,\varrho}|^p\,dx\right)^\frac{1}{p}\le C\,\frac{K(\varrho,u)}{\varrho^s}.
\]
By raising to power $p$ and integrating over $(0,+\infty)$ with respect to the singular measure $d\varrho/\varrho$, we get
\[
\int_0^{+\infty}\varrho^{-s\,p}\,\left(\int_{B_\varrho(x_0)} |u-u_{x_0,\varrho}|^p\,dx\right)\,\frac{d\varrho}{\varrho}\le C\,\int_{0}^{+\infty}\left(\frac{K(\varrho,u)}{\varrho^s}\right)^p\,\frac{d\varrho}{\varrho}.
\]
From \cite[Proposition 4.5]{Brasco2019}, we have that
\[
\int_{0}^{+\infty}\left(\frac{K(\varrho,u)}{\varrho^s}\right)^p\,\frac{d\varrho}{\varrho}\le \frac{C}{s\,(1-s)}\,[u]_{W^{s,p}(\mathbb{R}^n)}^p,
\]
for a constant $C=C(n,p)>0$. Up to now, we obtained
\begin{equation}
\label{proto}
\int_0^{+\infty}\varrho^{-s\,p}\,\left(\int_{B_\varrho(x_0)} |u-u_{x_0,\varrho}|^p\,dx\right)\,\frac{d\varrho}{\varrho}\le \frac{C}{s\,(1-s)}\,[u]_{W^{s,p}(\mathbb{R}^n)}^p.
\end{equation}
We now use that
\[
\int_{B_\varrho(x_0)} |u-u_{x_0,\varrho}|^p\,dx\ge \inf_{c\in\mathbb{R}} \int_{B_\varrho(x_0)} |u-c|^p\,dx,
\]
thus we get for $r>0$
\[
\begin{split}
\int_0^{+\infty}\varrho^{-s\,p}\,\left(\int_{B_\varrho(x_0)} |u-u_{x_0,\varrho}|^p\,dx\right)\,\frac{d\varrho}{\varrho}&\ge \int_r^{+\infty} \varrho^{-s\,p}\,\left(\inf_{c\in\mathbb{R}} \int_{B_\varrho(x_0)} |u-c|^p\,dx\right)\,\frac{d\varrho}{\varrho}\\
&\ge \left(\inf_{c\in\mathbb{R}} \int_{B_r(x_0)} |u-c|^p\,dx\right)\, \int_r^{+\infty} \varrho^{-s\,p}\,\,\frac{d\varrho}{\varrho}\\
&=\frac{1}{s\,p}\,r^{-s\,p}\,\left(\inf_{c\in\mathbb{R}} \int_{B_r(x_0)} |u-c|^p\,dx\right).
\end{split}
\]
Since $r>0$ and $x_0\in \mathbb{R}^n$ are arbitrary, from \eqref{proto} we thus obtain that
\[
\sup_{x_0\in\mathbb{R}^n,\, r>0}\left(r^{-s\,p}\,\inf_{c\in\mathbb{R}} \int_{B_r(x_0)} |u-c|^p\,dx\right)\le \frac{C}{1-s}\,[u]_{W^{s,p}(\mathbb{R}^n)}^p.
\]
By recalling that the quantity in the left-hand side is equivalent to the Campanato seminorm $\mathcal{L}^{p,sp}$  (see \cite[Remark 2.2]{Gi}), we get
\begin{equation}
\label{s}
[u]^p_{\mathcal{L}^{p,sp}(\mathbb{R}^n)}\le \frac{C}{1-s}\,[u]_{W^{s,p}(\mathbb{R}^n)}^p,
\end{equation}
for some $C=C(n,p)>0$.
In a similar way, we observe that for $r>0$
\[
\begin{split}
\int_0^{+\infty}\varrho^{-s\,p}\,\left(\int_{B_\varrho(x_0)} |u-u_{x_0,\varrho}|^p\,dx\right)\,\frac{d\varrho}{\varrho}&\ge \int_r^{2\,r} \varrho^{-s\,p}\,\left(\inf_{c\in\mathbb{R}} \int_{B_\varrho(x_0)} |u-c|^p\,dx\right)\,\frac{d\varrho}{\varrho}\\
&\ge \left(\inf_{c\in\mathbb{R}} \int_{B_r(x_0)} |u-c|^p\,dx\right)\, \int_r^{2\,r} \varrho^{-s\,p}\,\frac{d\varrho}{\varrho}\\
&=\left(\inf_{c\in\mathbb{R}} \int_{B_r(x_0)} |u-c|^p\,dx\right)\, \int_r^{2\,r} \frac{\varrho^{p-s\,p}}{\varrho^p}\,\frac{d\varrho}{\varrho}\\
&\ge \frac{1}{(2\,r)^p}\,\left(\inf_{c\in\mathbb{R}} \int_{B_r(x_0)} |u-c|^p\,dx\right)\, \int_r^{2\,r} \varrho^{p-s\,p}\,\frac{d\varrho}{\varrho} \\
&=\frac{1}{2^p\,(1-s)\,p}\,\frac{(2\,r)^{p-s\,p}-r^{p-s\,p}}{r^p}\,\left(\inf_{c\in\mathbb{R}} \int_{B_r(x_0)} |u-c|^p\,dx\right).
\end{split}
\]
As before, from \eqref{proto} we get
\[
\sup_{x_0\in\mathbb{R}^n,\, r>0}\left(r^{-s\,p}\,\inf_{c\in\mathbb{R}} \int_{B_r(x_0)} |u-c|^p\,dx\right)\le \frac{C}{s}\,[u]_{W^{s,p}(\mathbb{R}^n)}^p,
\]
and thus
\begin{equation}
\label{1-s}
[u]^p_{\mathcal{L}^{p,sp}(\mathbb{R}^n)}\le \frac{C}{s}\,[u]_{W^{s,p}(\mathbb{R}^n)}^p,
\end{equation}
for some $C=C(n,p)>0$. If we now multiply \eqref{s} by $(1-s)$, \eqref{1-s} by $s$ and then take the sum, we get the claimed estimate.
\end{proof}
As a first straightforward consequence of Theorem \ref{teo:campanatoholder}, we get the following
\begin{corollary}[Fractional Poincar\'e-Wirtinger inequality]
\label{coro:poincare}
Let $s\in(0,1)$ and $1\le p<+\infty$. Then for every $u\in C^\infty_c(\mathbb{R}^n)$ and every $x_0\in \mathbb{R}^n$, $R>0$, we have
\begin{equation}
\label{PW}
\int_{B_R(x_0)} |u-u_{x_0,R}|^p\,dx\le C\,R^{s\,p}\,[u]_{W^{s,p}(\mathbb{R}^n)}^p,
\end{equation}
for a constant $C=C(n,p)>0$.
\end{corollary}
\begin{remark}
\label{rem:pessetto}
As a simple consequence of the previous result, we also have the following more flexible inequality: for every $u\in C^\infty_c(\mathbb{R}^n)$ and every $x_0\in \mathbb{R}^n$, $0<r\le R$, we have
\begin{equation}
\label{PWbis}
\int_{B_R(x_0)} |u-u_{x_0,r}|^p\,dx\le C\,\left(1+\left(\frac{R}{r}\right)^n\right)\,R^{s\,p}\,[u]_{W^{s,p}(\mathbb{R}^n)}^p,
\end{equation}
for a possibly different constant $C=C(n,p)>0$.
Indeed, it is sufficient to observe that, due to Jensen's inequality,
\[
\begin{split}
\int_{B_R(x_0)} |u-u_{x_0,r}|^p\,dx&\le 2^{p-1}\, \int_{B_R(x_0)} |u-u_{x_0,R}|^p\,dx+2^{p-1}\,|B_R(x_0)|\,|u_{x_0,r}-u_{x_0,R}|^p\\
&\le 2^{p-1}\, \int_{B_R(x_0)} |u-u_{x_0,R}|^p\,dx+2^{p-1}\,\frac{|B_R(x_0)|}{|B_r(x_0)|}\,\int_{B_r(x_0)}|u-u_{x_0,R}|^p\,dx\\
&\le 2^{p-1}\,\left(1+\left(\frac{R}{r}\right)^n\right)\,\int_{B_R(x_0)}|u-u_{x_0,R}|^p\,dx.
\end{split}
\]
An application of Corollary \ref{coro:poincare} now leads to the claimed estimate \eqref{PWbis}.
\end{remark}
Theorem \ref{teo:campanatoholder}, together with the estimate \eqref{campanatoholder}, also implies the following result. We include in the statement the case $s=1$, which is classical, see for example \cite[Theorem 3.9]{Gi}.
\begin{corollary}[Fractional Morrey's inequality]
\label{coro:morrey}
Let $s\in(0,1]$ and $1\le p<+\infty$ be such that $s\,p> n$. Then for every $u\in C^\infty_c(\mathbb{R}^n)$ we have
\begin{equation}
\label{eq:Morrey}
[u]_{C^{0,s-\frac{n}{p}}(\mathbb{R}^n)}\le C\,[u]_{W^{s,p}(\mathbb{R}^n)},
\end{equation}
for a constant $C=C(n,p,s)>0$. Such a constant may be taken independent of $s$, whenever $s-n/p\ge \delta_0$, for some $\delta_0>0$. In this case, it has the form $C=C(n,p,\delta_0)>0$.
\end{corollary}
\begin{remark}
The previous result is well-known, see for example \cite[Th\'eor\`eme 8.2]{Pe} for a proof using a different interpolation-type argument. The main focus here is on the presence of the scaling factor $s\,(1-s)$, which is incorporated in our definition of the Gagliardo-Slobodecki\u{\i} seminorm. If one is not interested in keeping track of this factor, actually the proof simplifies, see for example \cite[Lemma 2.3]{Fr}.
\end{remark}
We conclude this section with a variant of the Poincar\'e-Wirtinger inequality. We do not pay too much attention to the quality of the constant: the resulting outcome will be sufficient for our purposes.
\begin{lemma}
\label{lm:gracias}
Let $s\in(0,1)$ and $1\le p<+\infty$. Let $u\in L^1_{\rm loc}(\mathbb{R}^n)$ be such that
\[
[u]_{W^{s,p}(\mathbb{R}^n)}<+\infty.
\]
The for every $0<r<R$ and $x_0\in\mathbb{R}^n$, we have
\[
\int_{B_R(x_0)\setminus B_r(x_0)} |u-u_{x_0,R}|^p\,dx\le C\,R^{s\,p}\,\iint_{(B_R(x_0)\setminus B_r(x_0))\times B_R(x_0)} \frac{|u(x)-u(y)|^p}{|x-y|^{n+s\,p}}\,dx\,dy,
\]
for a constant $C=C(n,p,s)>0$.
\end{lemma}
\begin{proof}
The proof is quite straightforward, it is the same that can be found in \cite[page 297]{Mi2003}, for example. By using Jensen's inequality, we have
\[
\begin{split}
\int_{B_R(x_0)\setminus B_r(x_0)} |u-u_{x_0,R}|^p\,dx\le \frac{1}{|B_R(x_0)|}\, \iint_{(B_R(x_0)\setminus B_r(x_0))\times B_{R}(x_0)} |u(x)-u(y)|^p\,dx\,dy.
\end{split}
\]
We now observe that
\[
1\le \frac{(2\,R)^{n+s\,p}}{|x-y|^{n+s\,p}},\qquad \mbox{ for a.\,e. }(x,y)\in B_R(x_0)\times  B_R(x_0).
\]
By using this simple fact in the previous estimate, we get the desired conclusion.
\end{proof}

\section{Characterisation for $s\,p < n$}\label{sec.spmn}
\label{sec:3}

This case is similar to the local case (i.\,e. $s=1$) for $p<n$.
Indeed, in this range we obtain an embedding into a Lebesgue space, by means of the {\it fractional Sobolev inequality}
\begin{equation}
\label{eq:GNS sp < n}
\mathcal{S}_{s,p}\,\| u \|^p_{L^{p^\star_s}(\mathbb{R}^n)} \le [u]^p_{W^{s,p} (\mathbb{R}^n)}, \qquad \mbox{ for every } u\in C^\infty_c(\mathbb{R}^n),
\end{equation}
for some constant $\mathcal{S}_{s,p}>0$. Here $p^\star_s$ is the critical Sobolev exponent, defined by
\[
p^\star_s = \frac{n\,p}{n-s\,p}.
\]
We refer to \cite[Theorem 10.2.1]{maz} for an elementary proof of \eqref{eq:GNS sp < n}. See also \cite[Th\'eor\`eme 8.1]{Pe} for an older proof, based on real interpolation techniques.

By using inequality \eqref{eq:GNS sp < n}, it is possible to give a concrete characterization of the completion $\mathcal{D}^{s,p}(\mathbb{R}^n)$ as a functional space.
\begin{theorem}
	\label{thm:sp < n}
Let $s \in (0,1]$ and  $1\le p<+\infty$ be such that $s\,p < n$. We indicate by $\dot W^{s,p}(\mathbb{R}^n)$ the space
\begin{equation}
\label{seteq}
	\dot W^{s,p} (\mathbb{R}^n) = \Big\{  u \in L^{p^\star_s} ( \mathbb{R}^n )\, :\,  [ u ]_{W^{s,p} (\mathbb{R}^n)} < +\infty  \Big\}.
\end{equation}
We endow this space with the norm
\[
\|u\|_{\dot W^{s,p} (\mathbb{R}^n)} = [u]_{W^{s,p} (\mathbb{R}^n)},\qquad \mbox{ for every }u\in \dot W^{s,p}(\mathbb{R}^n).
\]
Then this is a Banach space, having $C^\infty_c(\mathbb{R}^n)$ as a dense subspace. Moreover, there exists a linear isometric isomorphism
\[	
\mathcal{J}: \mathcal{D}^{s,p} (\mathbb{R}^n) \to   \dot W^{s,p} (\mathbb{R}^n).
\]
In other words, the space $\mathcal{D}^{s,p}(\mathbb{R}^n)$ can be identified with $\dot W^{s,p}(\mathbb{R}^n)$.	
\end{theorem}
\begin{proof}
It is easy to see that $\dot W^{s,p}(\mathbb{R}^n)$ is a normed vector space. The fact that this is a Banach space will follow from the claimed isometry, that we are going to construct at the end of the proof.
\par
We now divide the rest of the proof in three parts.
\medskip

\noindent
{\bf Part 1: density of smooth functions}. We prove that $C^\infty_c(\mathbb{R}^n)$ is dense in $\dot W^{s,p}(\mathbb{R}^n)$. We need to prove that for every
$ u \in \dot{W}^{s,p} (\mathbb{R}^n)$, there exists a sequence $(u_m)_{m\in\mathbb{N}}\subset C^\infty_c(\mathbb{R}^n)$ such that
\[
\lim_{m\to\infty} [u_m-u]_{W^{s,p}(\mathbb{R}^n)}=0.
\]
In order to construct the sequence $(u_m)_{m\in\mathbb{N}}$, we consider the sequence of smoothing kernels $(\rho_m)_{m\ge 1}$ as in the statement of Lemma \ref{lm:cazzata}. Moreover, we introduce a sequence of cut-off functions
$ \eta_j\in C^\infty_c (\mathbb R^n)$ with $\mathrm{supp\,} \eta_j \subset B_{2j}$ such that
\[
0\le \eta_j\le 1,\qquad \eta_j\equiv 1 \mbox{ on } B_j,\qquad |\nabla \eta_j|\le \frac{C}{j}.
\]
By Lemma \ref{lm:truncation}, for every $m\ge1$ we have
\[
\lim_{j\to\infty} [(u\ast\rho_m)\,\eta_j-u\ast\rho_m]_{W^{s,p}(\mathbb{R}^n)}=0.
\]
Thus, for every $m\ge 1$ we can choose $j_m\in\mathbb{N}$ such that
\[
[(u\ast\rho_m)\,\eta_{j_m}-u\ast\rho_m]_{W^{s,p}(\mathbb{R}^n)}\le \frac{1}{m}.
\]
We finally set
\[
u_m=(u\ast \rho_m)\,\eta_{j_m},
\]
then this sequence has the desired approximation property. Indeed, observe that by the triangle inequality we have
\[
\begin{split}
[u_m-u]_{W^{s,p}(\mathbb{R}^n)}&=[(u\ast\rho_m)\,\eta_{j_m}-u]_{W^{s,p}(\mathbb{R}^n)}\\
&\le [(u\ast\rho_m)\,\eta_{j_m}-u\ast \rho_m]_{W^{s,p}(\mathbb{R}^n)}+[u\ast\rho_m-u]_{W^{s,p}(\mathbb{R}^n)}\\
&\le \frac{1}{m}+[u\ast\rho_m-u]_{W^{s,p}(\mathbb{R}^n)}.
\end{split}
\]
By taking the limit as $m$ goes to $\infty$ and appealing to Lemma \ref{lm:cazzata}, we get the conclusion.
\medskip

\noindent
{\bf Part 2: Cauchy sequences in $\mathcal{D}^{s,p}(\mathbb{R}^n)$}.
We take a sequence $(u_m)_{m\in\mathbb{N}}\subset C^\infty_c(\mathbb{R}^n)$, which is a Cauchy sequence with respect to the norm
\[
\varphi\mapsto [\varphi]_{W^{s,p}(\mathbb{R}^n)}.
\]
By using the fractional Sobolev inequality \eqref{eq:GNS sp < n}, we have that this is a Cauchy sequence in $L^{p^\star_s}(\mathbb{R}^n)$, as well.  The latter being a Banach space, we get that the sequence converges strongly
 in $L^{p^\star_s}(\mathbb{R}^n)$ to a function $u\in L^{p^\star_s}(\mathbb{R}^n)$. Furthermore, we can show that $u\in \dot W^{s,p}(\mathbb{R}^n)$.
 \par
Indeed, if we fix $\varepsilon>0$, then by definition of Cauchy sequence there exists $n_\varepsilon\in\mathbb{N}$ such that
\[
[u_m-u_k]_{W^{s,p}(\mathbb{R}^n)}<\varepsilon,\qquad \mbox{ for every } k,m\ge n_\varepsilon.
\]
In particular, by Minkowski inequality we get
\[
\begin{split}
[u_m]_{W^{s,p}(\mathbb{R}^n)}&\le[u_m-u_{n_\varepsilon}]_{W^{s,p}(\mathbb{R}^n)}+[u_{n_\varepsilon}]_{W^{s,p}(\mathbb{R}^n)}\\
&<\varepsilon+[u_{n_\varepsilon}]_{W^{s,p}(\mathbb{R}^n)},\qquad \mbox{ for every }m\ge n_\varepsilon.
\end{split}
\]
This shows that the sequence
\begin{equation}
\label{cauchydoble}
\left(\frac{u_m(x)-u_m(y)}{|x-y|^{\frac{n}{p}+s}}\right)_{m\in\mathbb{N}}\subset L^p(\mathbb{R}^n\times \mathbb{R}^n),
\end{equation}
is bounded in $L^p(\mathbb{R}^n\times \mathbb{R}^n)$. By using that $u_m$ converges almost everywhere to $u$ (up to a subsequence), Fatou's Lemma entails
\[
[u]_{W^{s,p}(\mathbb{R}^n)}\le \liminf_{m\to\infty}[u_m]_{W^{s,p}(\mathbb{R}^n)}<+\infty,
\]
i.\,e. $u\in \dot W^{s,p}(\mathbb{R}^n)$.
\par
We now observe that $L^p(\mathbb{R}^n\times \mathbb{R}^n)$ is a Banach space, thus the Cauchy sequence \eqref{cauchydoble} converges strongly in $L^p(\mathbb{R}^n\times\mathbb{R}^n)$. By uniqueness, the limit must coincide with
\[
\frac{u(x)-u(y)}{|x-y|^{\frac{n}{p}+s}}.
\]
In conclusion, we obtain that the Cauchy sequence $(u_m)_{m\in\mathbb{N}}$ converges with respect to the Gagliardo-Slobodecki\u{\i} seminorm to an element of $\dot W^{s,p}(\mathbb{R}^n)$.
\medskip

\noindent
{\bf Part 3: construction of the isometry}. We now take $U\in \mathcal{D}^{s,p}(\mathbb{R}^n)$ and choose a representative of this equivalence class, i.e. $U=\{(u_m)_{m\in\mathbb{N}}\}_{s,p}$. Thanks to {\bf Part 2}, we know that $(u_m)_{m\in\mathbb{N}}$ converges to a function $u\in \dot W^{s,p}(\mathbb{R}^n)$. We then define
\[
\mathcal{J}(U)=u.
\]
Observe that this is well-defined, since for any other representative $(\widetilde u_m)_{m\in\mathbb{N}}$ belonging to the class $U$, we still have
\[
\lim_{m\to\infty}[\widetilde u_m-u]_{W^{s,p}(\mathbb{R}^n)}\le \lim_{m\to\infty} [\widetilde u_m-u_m]_{W^{s,p}(\mathbb{R}^n)}+\lim_{m\to\infty} [u_m-u]_{W^{s,p}(\mathbb{R}^n)}=0.
\]
Moreover, it is easy to see that $\mathcal{J}$ is linear. It is also immediate to obtain that this is an isometry, since
\[
\|U\|_{\mathcal{D}^{s,p}(\mathbb{R}^n)}=\lim_{m\to\infty} [u_m]_{W^{s,p}(\mathbb{R}^n)}=[u]_{W^{s,p}(\mathbb{R}^n)}=\|u\|_{\dot W^{s,p}}=\|\mathcal{J}(U)\|_{\dot W^{s,p}}.
\]
We are left with proving that $\mathcal{J}$ is surjective. From {\bf Part 1} we know that for every $u\in \dot W^{s,p}(\mathbb{R}^n)$
there exists a sequence $(u_m)_{n\in\mathbb{N}}\subset C^\infty_c(\mathbb{R}^n)$ such that
\[
\lim_{n\to\infty} [u_m-u]_{W^{s,p}(\mathbb{R}^n)}=0.
\]
In particular, this is a Cauchy sequence with respect to the Gagliardo-Slobodecki\u{\i} seminorm. Thus we get
\[
u=\mathcal{J}\Big(\{(u_m)_{m\in\mathbb{N}}\}_{s,p}\Big).
\]
This concludes the proof.
\end{proof}
\begin{remark}
We note that
\[
W^{s,p} (\mathbb{R}^n) \subset L^p (\mathbb{R}^n) \cap L^{p^\star_s} (\mathbb{R}^n),
\]
with continuous inclusion.
The inclusion in $L^p(\mathbb{R}^n)$ is a straightforward consequence of the definition \eqref{wsp} of $W^{s,p}(\mathbb{R}^n)$. On the other hand, the inclusion in $L^{p^\star_s} (\mathbb{R}^n)$ follows from the fractional Sobolev inequality and the density of $C^\infty_c(\mathbb{R}^n)$ functions in $W^{s,p}(\mathbb{R}^n)$.
\par
We can exploit this summability information to see that
\[
W^{s,p} (\mathbb{R}^n) \subsetneq \dot W^{s,p} (\mathbb{R}^n).
\]
For example, it is not difficult to see that the function
\[
\varphi(x)=(1+|x|^2)^{-\frac{\alpha}{2}},\qquad \mbox{ for } \frac{n}{p}-s<\alpha\le \frac{n}{p},
\]
is such that
\[
\varphi\in \dot W^{s,p} (\mathbb{R}^n)\setminus W^{s,p}(\mathbb{R}^n).
\]
Indeed, $\varphi\not\in L^p(\mathbb{R}^n)$ thanks to the choice of $\alpha$. In order to see that $\varphi$ has a finite Gagliardo-Slobodecki\u{\i} seminorm, it is useful to decompose the seminorm as follows
\[
\begin{split}
\frac{1}{s\,(1-s)}\,[\varphi]_{W^{s,p}(\mathbb{R}^n)}^p&=\iint_{(\mathbb{R}^n\setminus B_1(0))\times(\mathbb{R}^n\setminus B_1(0))} \frac{|\varphi(x)-\varphi(y)|^p}{|x-y|^{n+s\,p}}\,dx\,dy+\iint_{B_1(0) \times B_1(0)} \frac{|\varphi(x)-\varphi(y)|^p}{|x-y|^{n+s\,p}}\,dx\,dy\\
&+2\,\iint_{B_1(0)\times(\mathbb{R}^n\setminus B_2(0))} \frac{|\varphi(x)-\varphi(y)|^p}{|x-y|^{n+s\,p}}\,dx\,dy\\
&+2\,\iint_{B_1(0)\times (B_2(0)\setminus B_1(0))} \frac{|\varphi(x)-\varphi(y)|^p}{|x-y|^{n+s\,p}}\,dx\,dy=:\mathcal{I}_1+\mathcal{I}_2+\mathcal{I}_3+\mathcal{I}_4.
\end{split}
\]
The integrals $\mathcal{I}_2$ and $\mathcal{I}_4$ are finite, thanks to the Lipschitz character of $\varphi$. The third integral $\mathcal{I}_3$ is finite, by using that $\varphi\in L^\infty(\mathbb{R}^n)$ and that the function
\[
x\mapsto \int_{\mathbb{R}^n\setminus B_2(0)} \frac{1}{|x-y|^{n+s\,p}}\,dy,
\]
is uniformly bounded for $x\in B_1(0)$. Finally, for the finiteness of $\mathcal{I}_1$, it is sufficient to observe that\footnote{We can use the following fact: if we set
\[
f(\tau)=\frac{\tau}{(1+\tau^\frac{2}{\alpha})^\frac{\alpha}{2}},\qquad \mbox{ for } \tau\ge 0,
\]
then this is a Lipschitz function and we have
\[
\varphi(x)=(1+|x|^2)^{-\frac{\alpha}{2}}=f(|x|^{-\alpha}).
\]}
\[
|\varphi(x)-\varphi(y)|\le C\, \Big||x|^{-\alpha}-|y|^{-\alpha}\Big|,\qquad \mbox{ for every }x,y\in \mathbb{R}^n\setminus B_1(0),
\]
so that
\[
\mathcal{I}_1\le \iint_{(\mathbb{R}^n\setminus B_1(0))\times(\mathbb{R}^n\setminus B_1(0))} \frac{|\Psi(x)-\Psi(y)|^p}{|x-y|^{n+s\,p}}\,dx\,dy,\qquad \mbox{ with } \Psi(x)=|x|^{-\alpha}.
\]
The last double integral is then finite, by appealing to \cite[Lemma A.1]{BMS}.
\end{remark}
\begin{remark}
We take the occasion to recall that the interesting question of determining the sharp constant in \eqref{eq:GNS sp < n} is still open, except for the case $p=2$, solved in \cite{CT}. It is clear that the sharp constant is given by
	\begin{equation}
	\label{talenti}
		 \mathcal{S}_{s,p}= \inf_{u\in \mathcal{D}^{s,p}(\mathbb{R}^n)}\left\{[u]^p_{W^{s,p} (\mathbb{R}^n)}\, :\, \| u \|_{L^{p^\star_s}(\mathbb{R}^n)} =1\right\}.
	\end{equation}
	The relevant Euler-Lagrange optimality condition is a nonlinear eigenvalue-type equation involving the operator $(-\Delta_p)^s$, already presented in the Introduction. Namely, an extremal for the previous problem has to be a constant sign solution of
\[
(-\Delta_p)^s u=\mathcal{S}_{s,p}\,u^{p^\star_s-1},\qquad \mbox{ in } \mathbb{R}^n.
\]
Some properties of solutions to \eqref{talenti} have been investigated in \cite[Theorem 1.1]{BMS}.
\end{remark}

\section{Characterisation for $s\,p > n$}
\label{sec:4}
Instead of the fractional Sobolev inequality, in this range we have
the fractional Morrey inequality, see Corollary \ref{coro:morrey}.
However, unlike Sobolev's inequality, inequality \eqref{eq:Morrey} does not detect constants.
Even worse, in this range constant functions
can be approximated by sequences in  $C_c^\infty (\mathbb{R}^n)$ with respect to the Gagliardo-Slobodecki\u{\i} seminorm.

\begin{lemma}
	\label{lem:homogeneous sobolev counter}
	Let
$s \in (0,1]$ and	
$s\,p > n$.
There exists a sequence $(\varphi_m)_{m\in\mathbb{N}} \subset C_c^\infty (\mathbb{R}^n)$ such that
\begin{equation*}
 [ \varphi_m ] _{W^{s,p} (\mathbb{R}^n)} \le C\, m^{\frac n p - s} \to 0
 \qquad \text{ and } \qquad
 \varphi_m \to 1 \text{  uniformly over compact sets},
\end{equation*}
as $m\to +\infty$. Hence, $(\varphi_m)_{m\in\mathbb{N}}$ is equivalent in $\mathcal{D}^{s,p} (\mathbb{R}^n)$ to zero, although its pointwise limit is~$1$.
\end{lemma}
\begin{proof}
The proof is just based on the scaling properties of the Sobolev-Slobodecki\v{\i} seminorm, as in the local case.
Let $\varphi \in C_c^\infty (B_2)$ be a non-negative cut-off function, such that $\varphi$ coincides identically with $1$ on $B_1$. We define the rescaled sequence
\[
\varphi_m (x) = \varphi \left(\frac{x}{m}\right),\qquad \mbox{ for every }m\ge 1.
\]
By recalling \eqref{eq:scaling}, we have that
\nc
\[
[ \varphi_m ] _{W^{s,p} (\mathbb{R}^n)} = m^{\frac n p - s}\,  [ \varphi ] _{W^{s,p} (\mathbb{R}^n)}.
\]
The conclusion now follows, thanks to the fact that $n/p-s<0$.
\end{proof}
\begin{remark}
	Notice that this construction is linked with the {\it relative $(s,p)-$capacity} of a set $\omega\Subset \Omega$,	defined as
	\begin{equation*}
		(s,p) - \mathrm{cap}_\Omega (\omega) = \inf \left\{\iint_{\Omega \times \Omega} \frac{|\varphi(x)-\varphi(y)|^p}{|x-y|^{n+s\,p}}\,dx\,dy \, :\, \varphi \in C_c^\infty (\Omega),\, \varphi = 1 \text{ in } \omega \right\},
	\end{equation*}
	see for example \cite{Warma}. For $s=1$ this value can be explicitly computed, and the relevant Euler-Lagrange equation in linked to the usual $p-$Laplacian.
\end{remark}
We need the following technical result.
\begin{lemma}
	\label{lem:homog Sobolev sp > n}
	Let $s \in (0,1]$ and $s\,p > n$, we define $\alpha=s-n/p$. Let $(u_m)_{m\in\mathbb{N}} \subset C_c^\infty (\mathbb{R}^n)$ be Cauchy sequence with respect to $[\,\cdot\,]_{W^{s,p} (\mathbb{R}^n)}$. Then there exists another sequence $(\widetilde u_m)_{m\in\mathbb{N}}\subset C_c^\infty (\mathbb{R}^n)$ such that:
\begin{itemize}
\item we have	
	\[
	\lim_{m\to\infty}[u_m - \widetilde u_m]_{W^{s,p} (\mathbb{R}^n)}=0,
	\]
\item $\widetilde u_m$ converges uniformly over compact sets to a $\alpha-$H\"older continuous function $u$.
\end{itemize}
Moreover, this function $u$ is such that $u(0) = 0$,
	\begin{equation}
	\label{eq:estimates v_m}
		[\widetilde u_m - u]_{W^{s,p} (\mathbb{R}^n)} \to 0, \qquad 	[u]_{W^{s,p} (\mathbb{R}^n)} = \lim_{m\to\infty} [u_m]_{W^{s,p} (\mathbb{R}^n)},
	\end{equation}
	 and
	\begin{equation}
		\label{eq:homogeneous Sobolev Morrey}
		|u (x) - u(y)| \le C\, [u]_{W^{s,p} (\mathbb{R}^n)} |x-y|^\alpha, \qquad \mbox{ for every } x , y \in \mathbb{R}^n.
	\end{equation}
\end{lemma}
\begin{proof}
	We construct the sequence as follows. We take
	\begin{equation*}
	M_m \ge \max \left \{  m , \left(m\,|u_m (0)|\right)^{\frac p {s\,p - n}}   \right \},
	\end{equation*}
	and consider $\varphi_m$ given by Lemma \ref{lem:homogeneous sobolev counter}. By construction $M_m$ diverges to $\to +\infty$ as $m$ goes to  $+\infty$.
	Then we define
	\begin{equation*}
		\widetilde u_m (x) = u_m (x) - u_m (0)\, \varphi_{M_m} (x),\qquad \mbox{ for } m\in\mathbb{N}.
	\end{equation*}
We now show that $\widetilde{u}_m$ has the claimed properties. At first, by construction we have
\[
\begin{split}
[u_m - \widetilde u_m]_{W^{s,p} (\mathbb{R}^n)} = |u_m (0)|\, [\varphi_{M_m}]_{W^{s,p} (\mathbb{R}^n)}&\le C\,|u_m(0)|\,M_m^\frac{n-s\,p}{p}\\
&\le C\, |u_m(0)|\,\left(m\,|u_m (0)|\right)^{-1},
\end{split}
\]
which implies that
\[
\lim_{m\to\infty}[u_m - \widetilde u_m]_{W^{s,p} (\mathbb{R}^n)}=0,
\]
as desired. Observe that this implies that $(\widetilde u_m)_{m\in\mathbb{N}} \subset C_c^\infty (\mathbb{R}^n)$ is still a Cauchy sequence with respect to $[\,\cdot\,]_{W^{s,p} (\mathbb{R}^n)}$.

\medskip
In order to infer the uniform convergence, it is sufficient to show that $(\widetilde u_m)_{m\in\mathbb{N}}$ is a Cauchy sequence in $C(K)$, for every $K\subset \mathbb{R}^n$ compact set. This follows directly, by
applying \eqref{eq:Morrey} to $\widetilde{u}_m - \widetilde{u}_k$, i.\,e.
\[
|\widetilde u_m (x) - \widetilde u_k (x)| \le  C\, |x|^\alpha\, [\widetilde u_m - \widetilde u_k]_{W^{s,p} (\mathbb{R}^n)}.
\]
Hence, it converges uniformly to some function $u \in C(K)$. Since $K$ is arbitrary and the limit is unique, $u$ is defined for every $x \in \mathbb{R}^n$. Moreover, it holds $u(0)=0$, since we have $\widetilde u_m(0)=0$ by construction.
\medskip

By using that $(\widetilde u_m)_{m\in\mathbb{N}} \subset C_c^\infty (\mathbb{R}^n)$ is still a Cauchy sequence with respect to $[\,\cdot\,]_{W^{s,p} (\mathbb{R}^n)}$ and arguing as in {\bf Part 3} of the proof of Theorem \ref{thm:sp < n}, we deduce \eqref{eq:estimates v_m}. 	
The estimate \eqref{eq:homogeneous Sobolev Morrey} can then be obtained by passing to the limit in \eqref{eq:Morrey}.
\end{proof}

The major difference with respect to the case $s\,p<n$ is that now the elements in $\mathcal{D}^{s,p} (\mathbb{R}^n)$ {\it can not} be uniquely represented by functions. Indeed, when $s\,p > n$, any sequence $(u_m)_{m\in\mathbb{N}}\subset C_c^\infty (\mathbb{R}^n)$ which is Cauchy in the norm $[\,\cdot\,]_{W^{s,p} (\mathbb{R}^n)}$ is equivalent to the sequence
 \[
 v_m = u_m + C\, \varphi_m,
 \]
for any constant $C\in\mathbb{R}$. Here $\varphi_m$ is the same as in Lemma \ref{lem:homogeneous sobolev counter}.

\medskip
However, one can show that functions that are approximated by equivalent Cauchy sequences, actually coincide up to a constant. In other words, the homogeneous space $\mathcal{D}^{s,p}(\mathbb{R}^n)$ can be identified with a space of equivalence classes of functions differing by an additive constant.

\medskip
More precisely, we have the following characterization, which is the main result of this section.
\begin{theorem}
\label{thm:sp > n}
Let $s \in (0,1]$ and $s\,p > n$, we set $\alpha= s -n/p$. We consider the quotient space
\begin{equation*}
\dot W^{s,p} (\mathbb{R}^n) \vcentcolon=  \frac{\Big\{  u \in C^{0,\alpha}(\mathbb{R}^n)\, :\, [ u ]_{W^{s,p} (\mathbb{R}^n)} < +\infty  \Big\}}{\sim_C},
\end{equation*}
where
\begin{equation*}
\label{eq:equivalence relation up to constant}
u\sim_C v \qquad \Longleftrightarrow\qquad u-v \mbox{ is constant}.
\end{equation*}
We will denote by $\{ u\}_{C}$ the equivalence class of $u$ with respect to this relation.
We endow this space with the norm
\[
\|\{u\}_C\|_{\dot W^{s,p} (\mathbb{R}^n)} = [u]_{W^{s,p} (\mathbb{R}^n)},\qquad \mbox{ for every }u\in C^{0,\alpha}(\mathbb{R}^n) \mbox{ such that }[ u ]_{W^{s,p} (\mathbb{R}^n)} < +\infty .
\]
Then this is a Banach space and there exists a linear isometric isomorphism
	\begin{equation*}
	\mathcal{J}: \mathcal{D}^{s,p} (\mathbb{R}^n) \to   \dot W^{s,p} (\mathbb{R}^n).
	\end{equation*}
In other words, the space $\mathcal{D}^{s,p}(\mathbb{R}^n)$ can be identified with $\dot W^{s,p}(\mathbb{R}^n)$.	
\end{theorem}

\begin{proof}
It is immediate to see that $\dot W^{s,p} (\mathbb{R}^n)$ is a normed vector space. Indeed, constant functions all belong to the equivalence class $\{0\}_C$.
\par
We now construct the isometry. The fact that $\dot W^{s,p} (\mathbb{R}^n)$ is a Banach space will follow at once.
For any class
\[
U=\{(u_m)_{m\in\mathbb{N}}\}_{s,p}\in \mathcal{D}^{s,p}(\mathbb{R}^n),
\]
we may apply Lemma \ref{lem:homog Sobolev sp > n} and consider the new Cauchy sequence $(\widetilde u_m)_{m\in\mathbb{N}}$. By construction, we have
\[
U=\{(u_m)_{m\in\mathbb{N}}\}_{s,p}=\{(\widetilde u_m)_{m\in\mathbb{N}}\}_{s,p},
\]
and by Lemma \ref{lem:homog Sobolev sp > n}, we know that $(\widetilde u_m)_{m\in\mathbb{N}}$ converges to some function
\[
u\in \Big\{  \varphi \in C^{0,\alpha}(\mathbb{R}^n)\, :\, [ \varphi ]_{W^{s,p} (\mathbb{R}^n)} < +\infty  \Big\}.
\]
We may identify $U$ with the equivalence class $\{u\}_C$, i.\,e. we define $\mathcal{J}(U) =\{u\}_C$.
\par
Observe that this is well-defined, since for any other representative $(v_m)_{m\in\mathbb{N}}$ belonging to the class $U$, we still have
\[
\lim_{m\to\infty}[v_m-u]_{W^{s,p}(\mathbb{R}^n)}\le \lim_{m\to\infty} [\widetilde u_m-v_m]_{W^{s,p}(\mathbb{R}^n)}+\lim_{m\to\infty} [\widetilde u_m-u]_{W^{s,p}(\mathbb{R}^n)}=0.
\]
Let us show that the map $\mathcal{J}$ is linear. Indeed for every $U,V\in \mathcal{D}^{s,p}(\mathbb{R}^n)$ and $\alpha,\beta\in\mathbb{R}$, we consider $\alpha\,U+\beta\,V$. By choosing $(u_m)_{m\in\mathbb{N}}$ a representative for $U$ and $(v_m)_{m\in\mathbb{N}}$ a representative for $V$, we can apply Lemma \ref{lem:homog Sobolev sp > n} to both sequences and obtain two new sequences $(\widetilde u_m)_{m\in\mathbb{N}}, (\widetilde v_m)_{m\in\mathbb{N}}$
\[
(\widetilde u_m)_{m\in\mathbb{N}}\sim_{s,p} (u_m)_{m\in\mathbb{N}}\qquad \mbox{ and }\qquad (\widetilde v_m)_{m\in\mathbb{N}}\sim_{s,p} (v_m)_{m\in\mathbb{N}},
\]
and two functions
\[
u,v\in \Big\{  \varphi \in C^{0,\alpha}(\mathbb{R}^n)\, :\, [ \varphi ]_{W^{s,p} (\mathbb{R}^n)} < +\infty  \Big\},
\]
such that we have
\[
\mathcal{J}(U)=\{u\}_C\qquad \mbox{ and } \qquad \mathcal{J}(V)=\{v\}_C.
\]
By observing that
\[
\alpha\,U +\beta\,V =\{(\alpha\,\widetilde u_m+\beta\,\widetilde v_m)\}_{s,p},
\]
and using that
\[
\lim_{m\to\infty} [(\alpha\,\widetilde u_m+\beta\,\widetilde v_m)-(\alpha\,u+\beta\,v)]_{W^{s,p}(\mathbb{R}^n)}=0,
\]
we get
\[
\mathcal{J}(\alpha\,U+\beta\,V)=\{\alpha\,u+\beta\,v\}_C=\alpha\,\{u\}_C+\beta\,\{v\}_C=\alpha\,\mathcal{J}(U)+\beta\,\mathcal{J}(V),
\]
as desired.
\par
Moreover, by construction we have
\[
\|\mathcal{J}(U)\|_{\dot W^{s,p}(\mathbb{R}^n)}=\|\{u\}_C\|_{\dot W^{s,p}(\mathbb{R}^n)}=[u]_{W^{s,p}(\mathbb{R}^n)}=\lim_{m\to\infty} [\widetilde u_m]_{W^{s,p}(\mathbb{R}^n)}=\|U\|_{\mathcal{D}^{s,p}(\mathbb{R}^n)},
\]
which implies that this is an isometry.
\par
We still have to show that $\mathcal{J}$ is surjective. For every equivalence class $\{v\}_C\in \dot W^{s,p}(\mathbb{R}^n)$, we may select the representative $v$ in such a way that
\[
v(0)=0.
\]
Then we can construct a sequence $(v_m)_{m\in\mathbb{N}}\subset C^\infty_c(\mathbb{R}^n)$ such that
\begin{equation}
\label{approssima2}
\lim_{m\to\infty} [v_m-v]_{W^{s,p}(\mathbb{R}^n)}=0.
\end{equation}
In order to do this, we can repeat the construction of {\bf Part 1} in the proof of Theorem \ref{thm:sp < n}, up to some modifications that we are going to detail. More precisely, we introduce  a sequence of cut-off functions
$ \eta_j\in C^\infty_c (\mathbb R^n)$ with $\mathrm{supp\,} \eta_j \subset B_{2j}(0)$
such that
\[
0\le \eta_j\le 1,\qquad \eta_j\equiv 1 \mbox{ on } B_j(0),\qquad |\nabla \eta_j|\le \frac{C}{j},
\]
and observe that by Lemma \ref{lm:truncation2}, for every $m\ge1$
\[
[(u\ast\rho_m)\,\eta_j-u\ast\rho_m]_{W^{s,p}(\mathbb{R}^n)}\le C,
\]
for a constant $C>0$ independent of $j$. This shows that the sequence
\begin{equation}
\label{mah!}
\left(\frac{(u\ast\rho_m(x))\,(\eta_j(x)-1)-(u\ast\rho_m(y))\,(\eta_j(y)-1)}{|x-y|^{\frac{n}{p}+s}}\right)_{j\in\mathbb{N}}\subset L^p(\mathbb{R}^n\times\mathbb{R}^n),
\end{equation}
weakly converges, up to a subsequence. The weak limit is given by the null function, since $1-\eta_j$ converges to $0$, locally uniformly. In order to upgrade this convergence, we can apply Mazur's Lemma to infer that there exists a new sequence made of convex combinations of \eqref{mah!}, which converges strongly to $0$ in $L^p(\mathbb{R}^n\times\mathbb{R}^n)$.
\par
Thanks to the form of the sequence \eqref{mah!}, this finally implies that there exists $(\widetilde\eta_j)_{j\in\mathbb{N}}\subset C^\infty_c(\mathbb{R}^n)$, such that
\[
\lim_{j\to\infty} \left\|\frac{(u\ast\rho_m(x))\,(\widetilde\eta_j(x)-1)-(u\ast\rho_m(y))\,(\widetilde\eta_j(y)-1)}{|x-y|^{\frac{n}{p}+s}}\right\|_{L^p(\mathbb{R}^n\times\mathbb{R}^n)}=0.
\]
in other words, we have
\[
\lim_{j\to\infty} [(u\ast\rho_m)\,\widetilde\eta_j-u\ast\rho_m]_{W^{s,p}(\mathbb{R}^n)}=0.
\]
Thus, for every $m\ge 1$ we can choose $j_m\in\mathbb{N}$ such that
\[
[(u\ast\rho_m)\,\widetilde\eta_{j_m}-u\ast\rho_m]_{W^{s,p}(\mathbb{R}^n)}\le \frac{1}{m}.
\]
If we now define the sequence $(v_m)_{m\in\mathbb{N}}\subset C^\infty_c(\mathbb{R}^n)$ by
\[
v_m=(u\ast\rho_m)\,\widetilde\eta_{j_m},
\]
it is easy to see that this verifies \eqref{approssima2}, thanks to the choice of $j_m$ and Lemma \ref{lm:cazzata}.
\par
Thus we can identify $\{v\}_C$ with the equivalence class $\{(v_m)_{m\in\mathbb{N}}\}_{s,p}$. In other words, this proves the surjectivity of $\mathcal{J}$. The proof is over.
\end{proof}

 \begin{remark} We recall that we indicate
\[
C^{0,\alpha}(\mathbb{R}^n)=\left\{u:\mathbb{R}^n\to\mathbb{R}\,:\, \sup_{x\not=y}\frac{|u(x)-u(y)|}{|x-y|^\alpha} < +\infty\right\},
\]
thus the functions belonging to this space are not necessarily bounded.
\end{remark}

\section{Characterisation for $s\,p = n$}
\label{sec:5}

\subsection{General case}\label{sec.spn}

We start with the corresponding version Lemma \ref{lem:homogeneous sobolev counter} for the case $s\,p=n$. The construction now is slightly more complicated, in particular it cannot be an easy consequence of scalings,  as we have explained in the Introduction. Hence, a careful choice of auxiliary functions is needed.
\begin{lemma}
	\label{lem:confo}
	Let $s \in (0,1]$ and $n \ge 1$ be such that $s<n$. There exists a sequence $(\varphi_m)_{m\in\mathbb{N}} \subset C_c^\infty (\mathbb{R}^n)$ such that
	\begin{equation*}
	[ \varphi_m ] _{W^{s,\frac{n}{s}} (\mathbb{R}^n)}\le C\, \left(\frac{1}{\log m}\right)^{1-\frac{s}{n}}\to 0
	\qquad \text{ and } \qquad
	\varphi_m \to 1 \text{  uniformly over compact sets},
	\end{equation*}
	as $m\to +\infty$. Hence, $(\varphi_m)_{m\in\mathbb{N}}$ is equivalent in $\mathcal{D}^{s,\frac{n}{s}} (\mathbb{R}^n)$ to zero, although its pointwise limit is $1$.
\end{lemma}
\begin{proof}
	As in \cite[page 319]{DL} and \cite[Lemma 15.2.2]{maz}, we take the sequence
	\begin{equation}
	\label{eq:null sequence s in (0,1) and sp = n}
	\psi_m (x)  = \left\{\begin{array}{ll}
	1, & \text{ if } |x| < m, \\
	&\\
	\dfrac{1}{\log m} \log \dfrac{m^2}{|x|}, & \text{ if } m\le |x|\le m^2, \\
	&\\
	0, & \text{ if } |x| > m^2.
	\end{array}
	\right.
	\end{equation}
	For $s=1$, it can be checked by direct computation that
	\[
	\lim_{m\to\infty} \| \nabla \psi_m \|_{L^n(\mathbb{R}^n)}=0.
	\]
	For the fractional case $s\in(0,1)$, we claim that we still have
	\[
	\lim_{m\to\infty} [\psi_m]_{W^{s,\frac{n}{s}}(\mathbb{R}^n)}=0,
	\]
	but the direct computation is fairly more intricate. We introduce the sequence
\[
u_m(x)=\left\{\begin{array}{ll}
	\left|\log \dfrac{1}{m}\right|^\frac{n-s}{n}, & \text{ if } |x| < \dfrac{1}{m}, \\
	&\\
	\left|\log \dfrac{1}{m}\right|^{-\frac{s}{n}}\,\log |x|, & \text{ if } \dfrac{1}{m}\le |x|\le 1, \\
	&\\
	0, & \text{ if } |x| > 1,
	\end{array}
	\right.	
\]
and observe that this is related to $\psi_m$ through the relation
\[
\psi_m(x)=\left|\log \dfrac{1}{m}\right|^\frac{s-n}{n}\,u_m\left(\frac{x}{m^2}\right).
\]
By using that the Gagliardo-Slobodecki\u{\i} seminorm is now scale invariant, we thus get
\[
[ \psi_m ] _{W^{s,\frac{n}{s}} (\mathbb{R}^n)}=\left|\log \dfrac{1}{m}\right|^\frac{s-n}{n}\, [u_m ] _{W^{s,\frac{n}{s}} (\mathbb{R}^n)}.
\]
We now recall that from \cite[Proposition 5.1]{Parini2019} we have
\[
\lim_{m\to\infty} [u_m ] _{W^{s,\frac{n}{s}} (\mathbb{R}^n)}=\gamma_{n,s}<+\infty.
\]	
We finally get that
	\[
	\lim_{m\to\infty} [\psi_m]_{W^{s,\frac{n}{s}}(\mathbb{R}^n)}\le \lim_{m\to\infty} C\,\left(\frac{1}{\log m}\right)^{1-\frac{s}{n}}=0,
	\]
	as claimed.
	\par
	Observe that technically speaking such a sequence $(\psi_m)_{m\in\mathbb{N}}$ does not belong to $C^\infty_c (\mathbb{R}^n)$. However, this is a minor issue, that can be easily sorted by convolution. We take $\rho\in C^\infty_0(\mathbb{R}^n)$ a standard Friedrichs mollifier supported on $B_1(0)$, then by defining
	\[
	\varphi_m=\psi_m\ast \rho,
	\]
	we get the desired conclusion, thanks to the properties of convolutions.
\end{proof}
The next technical result is the counterpart of Lemma \ref{lem:homog Sobolev sp > n} for the case $s\,p=n$. In particular, the space of H\"older functions now has to be replaced by a suitable Campanato space.

\begin{lemma}
	\label{lem:homog Sobolev sp = n}
	Let $s \in (0,1]$ and $n\ge 1$ be such that $s<n$. Let $(u_m)_{m\in\mathbb{N}} \subset C_c^\infty (\mathbb{R}^n)$ be Cauchy sequence with respect to $[\,\cdot\,]_{W^{s,n/s} (\mathbb{R}^n)}$. Then there exists another sequence $(\widetilde u_m)_{m\in\mathbb{N}}\subset C_c^\infty (\mathbb{R}^n)$ such that:
	\begin{itemize}
		\item we have	
		\[
		\lim_{m\to\infty}[u_m - \widetilde u_m]_{W^{s,\frac{n}{s}} (\mathbb{R}^n)}=0,
		\]
		\item $\widetilde u_m$ converges in $L^\frac{n}{s}_{\rm loc}(\mathbb{R}^n)$ to a function $u\in\mathcal{L}^{\frac{n}{s},n}(\mathbb{R}^n)$.
	\end{itemize}
	Moreover, this function $u$ is such that
	\[
	\int_{B_1(0)} u\,dx=0,
	\]
	and we have
	\begin{equation}
	\label{eq:estimates v_m conf}
	\lim_{m\to\infty}[\widetilde u_m - u]_{W^{s,\frac{n}{s}} (\mathbb{R}^n)}= 0, \qquad 	[u]_{W^{s,\frac{n}{s}} (\mathbb{R}^n)} = \lim_{m\to\infty} [u_m]_{W^{s,\frac{n}{s}} (\mathbb{R}^n)},
	\end{equation}
	\begin{equation}
	\label{eq:homogeneous Sobolev Morrey conf}
	[u]_{\mathcal{L}^{\frac{n}{s},n}(\mathbb{R}^n)} \le C\, [u]_{W^{s,\frac{n}{s}} (\mathbb{R}^n)}.
	\end{equation}
\end{lemma}
\begin{proof}
	The proof is similar to that of Lemma \ref{lem:homog Sobolev sp > n}. For every $m\in\mathbb{N}$, we choose a natural number $M_m\ge m$ large enough, so that
	\[
	\left|\int_{B_1(0)} u_m\,dx\right|\,\left(\frac{1}{\log M_m}\right)^{1-\frac{s}{n}}\le \frac{1}{m},
	\]
	and consider the sequence $(\varphi_m)_{m\in\mathbb{N}}$ given by Lemma \ref{lem:confo}. By construction we have that $M_m$ diverges to $+\infty$, as $m$ goes to $\infty$.
	Then we define
	\begin{equation*}
	\widetilde u_m (x) = u_m (x) - \frac{1}{|B_1(0)|}\,\left(\int_{B_1(0)} u_m\,dx\right) \varphi_{M_m} (x),\qquad \mbox{ for } m\in\mathbb{N}.
	\end{equation*}
	It is not difficult to see that
	\begin{equation}
	\label{vab}
	\int_{B_1(0)} \widetilde u_m\,dx=0.
	\end{equation}
	We now show that $\widetilde u_m$ has the claimed properties. Thanks to the choice of $M_m$ and Lemma \ref{lem:confo}, we have
	\[
	\begin{split}
	[u_m - \widetilde u_m]_{W^{s,\frac{n}{s}} (\mathbb{R}^n)} = \frac{1}{|B_1(0)|}\,\left|\int_{B_1(0)} u_m\,dx\right|\, [\varphi_{M_m}]_{W^{s,\frac{n}{s}} (\mathbb{R}^n)}&\le C\,\left|\int_{B_1(0)} u_m\,dx\right|\,\left(\frac{1}{\log M_m}\right)^{1-\frac{s}{n}}\le \frac{C}{m},
	\end{split}
	\]
	which implies that
	\[
	\lim_{m\to\infty}[u_m - \widetilde u_m]_{W^{s,\frac{n}{s}} (\mathbb{R}^n)}=0,
	\]
	as desired. In particular, we get that $(\widetilde u_m)_{m\in\mathbb{N}} \subset C_c^\infty (\mathbb{R}^n)$ is still a Cauchy sequence with respect to the seminorm $[\,\cdot\,]_{W^{s,n/s} (\mathbb{R}^n)}$.
	\par
	In order to infer the claimed convergence, we can apply Remark \ref{rem:pessetto} with $x_0=0$ and $r=1$, so to get that $(\widetilde u_m)_{m\in\mathbb{N}}$ is a Cauchy sequence in $L^{n/s}(B_R)$, for every $R\ge 1$. Thus we get that there exists $u\in L^{n/s}_{\rm loc}(\mathbb{R}^n)$ such that
	\[
	\lim_{m\to\infty} \|\widetilde u_m-u\|_{L^\frac{n}{s}(B_R)}=0,\qquad \mbox{ for every } R\ge 1.
	\]
	In particular, the strong convergence entails that
	\[
	\int_{B_1(0)} u\,dx=0,\qquad u\in \mathcal{L}^{\frac{n}{s},n}(\mathbb{R}^n)\qquad \mbox{ and }\qquad [u]_{W^{s,\frac{n}{s}}(\mathbb{R}^n)}<+\infty.
	\]
	The first property is straightforward, by recalling \eqref{vab}. The fact that $u\in \mathcal{L}^{\frac{n}{s},n}(\mathbb{R}^n)$ follows from
	\[
	\begin{split}
	\int_{B_\varrho(x_0)} |u-u_{x_0,\varrho}|^\frac{n}{s}\,dx=\lim_{m\to\infty} \int_{B_\varrho(x_0)} |\widetilde{u}_m-(\widetilde{u}_m)_{x_0,\varrho}|^\frac{n}{s}\,dx&\le  \varrho^n\, \lim_{m\to\infty} [\widetilde u_m]^\frac{n}{s}_{\mathcal{L}^{\frac{n}{s},n}(\mathbb{R}^n)}\\
	&\le C\, \varrho^n\,\lim_{m\to\infty} [\widetilde u_m]^\frac{n}{s}_{W^{s,\frac{n}{s}}(\mathbb{R}^n)},
	\end{split}
	\]
	where we used Theorem \ref{teo:campanatoholder}. This permits to infer that
	\[
	\sup_{x_0\in\mathbb{R}^n, \varrho>0}\varrho^{-n}\,\int_{B_\varrho(x_0)} |u-u_{x_0,\varrho}|^\frac{n}{s}\,dx<+\infty,
	\]
	as claimed. Finally, the fact that
	\[
	[u]_{W^{s,\frac{n}{s}}(\mathbb{R}^n)}<+\infty,
	\]
	follows from the lower semicontinuity of the Gagliardo-Slobodecki\u{\i} seminorm with respect to the strong $L^{n/s}$ convergence, which in turn follows from Fatou's Lemma.
	\par
	By using that $(\widetilde u_m)_{m\in\mathbb{N}} \subset C_c^\infty (\mathbb{R}^n)$ is still a Cauchy sequence with respect to $[\,\cdot\,]_{W^{s,p} (\mathbb{R}^n)}$ and arguing as in {\bf Part 2} of the proof of Theorem \ref{thm:sp < n}, we deduce \eqref{eq:estimates v_m conf}. 	
	The estimate \eqref{eq:homogeneous Sobolev Morrey conf} can then be obtained by passing to the limit in the inequality of Theorem \ref{teo:campanatoholder}.
	\end{proof}

\begin{theorem}
	\label{thm:sp = n}
	Let $s \in (0,1]$ and $n \ge 1$ be such that $s<n$.
	We define the quotient space
		\[
	\dot W^{s, \frac n s } (\mathbb{R}^n)\vcentcolon=  \frac{\left \{  v \in \mathcal{L}^{\frac{n}{s},n} (\mathbb{R}^n)\, :\, [v]_{W^{s,\frac n s } (\mathbb{R}^n)} < +\infty  \right \}}{\sim_C},
	\]
	where $\sim_C$ is the same equivalence relation as in \eqref{eq:equivalence relation up to constant}.
	We still endow this space with the norm
	\[
	\left\|\{u\}_C \right\|_{\dot W^{s,\frac n s } (\mathbb{R}^n)} = [u]_{W^{s,\frac n s } (\mathbb{R}^n)},\qquad \mbox{ for every }u\in \mathcal{L}^{\frac{n}{s},n}(\mathbb{R}^n) \mbox{ such that }[ u ]_{W^{s,\frac{n}{s}} (\mathbb{R}^n)} < +\infty.
	\]
	Then $\dot W^{s,\frac n s } (\mathbb{R}^n) $ is a Banach space and there exists a linear isometric isomorphism
	\begin{equation*}
	\mathcal{J}: \mathcal{D}^{s,\frac n s } (\mathbb{R}^n) \to	\dot W^{s,\frac n s  } (\mathbb{R}^n).
	\end{equation*}	
	In other words, the space $\mathcal{D}^{s,\frac{n}{s}}(\mathbb{R}^n)$ can be identified with $\dot W^{s,\frac{n}{s}}(\mathbb{R}^n)$.	
\end{theorem}

\begin{proof}
	The proof goes along the same lines of that of Theorem \ref{thm:sp > n}. The fact that $\dot W^{s,n/s} (\mathbb{R}^n)$ is a normed vector space is straightforward.
\par
Let us now consider the identification of our space.
	We take an equivalence class
	\[
	U=\{(u_m)_{m\in\mathbb{N}}\}_{s,p}\in \mathcal{D}^{s,p}(\mathbb{R}^n),
	\]
	and apply Lemma \ref{lem:homog Sobolev sp = n}, so to get the new sequence $(\widetilde u_m)_{m\in\mathbb{N}}$. By construction, we have
	\[
	U=\{(u_m)_{m\in\mathbb{N}}\}_{s,p}=\{(\widetilde u_m)_{m\in\mathbb{N}}\}_{s,p},
	\]
	and by Lemma \ref{lem:homog Sobolev sp = n}, we know that $(\widetilde u_m)_{m\in\mathbb{N}}$ converges to some function
	\[
	u\in \Big\{  \varphi \in \mathcal{L}^{\frac{n}{s},n}(\mathbb{R}^n)\, :\, [ \varphi ]_{W^{s,\frac{n}{s}} (\mathbb{R}^n)} < +\infty  \Big\}.
	\]
	Thus we may identify $U$ with the equivalence class $\{u\}_C$, i.\,e. we define $\mathcal{J}(U) =\{u\}_C$. As in the case $s\,p>n$, it is easily seen that this is a linear isometry.
	\par
	On the other hand, for every equivalence class $\{v\}_C\in \dot W^{s,p}(\mathbb{R}^n)$, we may select any representative $v$. 	Then we can construct a sequence $(v_m)_{m\in\mathbb{N}}\subset C^\infty_c(\mathbb{R}^n)$ such that
	\[
	\lim_{m\to\infty} [v_m-v]_{W^{s,p}(\mathbb{R}^n)}=0.
	\]
	It is indeed sufficient to repeat the construction of {\bf Part 1} in the proof of Theorem \ref{thm:sp < n}, with some minor modifications. We consider a sequence of cut-off functions
	$ \eta_j\in C^\infty_c (\mathbb R^n)$ with $\mathrm{supp\,} \eta_j \subset B_{j^2}$
	such that
\[
0\le \eta_j\le 1,\qquad \eta_j\equiv 1 \mbox{ on } B_j,\qquad |\nabla \eta_j|\le \frac{C}{j^2},
\]
and observe that by Lemma \ref{lm:truncation3}, we have
\[
\lim_{j\to\infty}[(v\ast\rho_m-\overline{v}_{j,m})\,\eta_j-(v\ast\rho_m-\overline{v}_{j,m})]_{W^{s,p}(\mathbb{R}^n)}=0.
\]
Here we used the shortcut notation
\[
\overline{v}_{j,m}=\frac{1}{|B_{j^2}(0)|}\,\int_{B_{j^2}(0)} v\ast\rho_m\,dx.
\]
Thus, for every $m\ge 1$ we can choose $j_m\in\mathbb{N}$ such that
\[
[(v\ast\rho_m-\overline{v}_{j_m,m})\,\eta_{j_m}-(v\ast\rho_m-\overline{v}_{j_m,m})]_{W^{s,p}(\mathbb{R}^n)}\le \frac{1}{m}.
\]
The sequence $(v_m)_{m\in\mathbb{N}}\subset C^\infty_c(\mathbb{R}^n)$ is then given by
\[
v_m=(v\ast\rho_m-\overline{v}_{j_m,m})\,\eta_{j_m}.
\]
Indeed, by construction we have
\[
\begin{split}
[v_m-v]_{W^{s,\frac{n}{s}}(\mathbb{R}^n)}&= [v_m-(v-\overline{v}_{j_m,m})]_{W^{s,\frac{n}{s}}(\mathbb{R}^n)}\\
&\le [(v\ast\rho_m-\overline{v}_{j_m,m})\,\eta_{j_m}-(v\ast\rho_m-\overline{v}_{j_m,m})]_{W^{s,p}(\mathbb{R}^n)}\\
&+[(v\ast\rho_m-\overline{v}_{j_m,m})-(v-\overline{v}_{j_m,m})]_{W^{s,p}(\mathbb{R}^n)}\\
&\le \frac{1}{m}+[(v\ast\rho_m-\overline{v}_{j_m,m})-(v-\overline{v}_{j_m,m})]_{W^{s,p}(\mathbb{R}^n)}\\
&=\frac{1}{m}+[v\ast\rho_m-v]_{W^{s,p}(\mathbb{R}^n)}.
\end{split}
\]
Thus, by Lemma \ref{lm:cazzata} again, we get
\[
\lim_{m\to\infty} [v_m-v]_{W^{s,p}(\mathbb{R}^n)}=0,
\]
as claimed. Accordingly, we can identify $\{v\}_C$ with the equivalence class $\{(v_m)_{m\in\mathbb{N}}\}_{s,p}$. In other words, this proves the surjectivity of $\mathcal{J}$. The proof is complete.\end{proof}

\begin{remark}
	In the case $s\,p=n$, a natural replacement for the Sobolev inequality is the {\it Moser-Trudinger inequality}. However, this holds for open sets with finite measure, see e.g. \cite{Parini2019}.
\end{remark}

\subsection{The exceptional limit case $s=n=1$}

Observe that in the previous section we have the restriction $s<n$. Thus, in order to complete the picture in the conformal case, there is still a limiting case which is missing: the case $s=1=n$. Accordingly, the summability exponent is $p=1$, as well.
This one-dimensional case is special
and deserves to be treated separately.
\par
As we will see, this situation is similar to the case $s\,p < n$. Indeed, we have the following result.
\begin{theorem}[The case $s=p=n=1$]
\label{teo:spn1}
Let us define
\[	
\dot W^{1,1} (\mathbb{R})  \left\{  u  \in C_0 (\mathbb{R})\,  :\, u'\in L^1(\mathbb{R}) \right \},
\]
where the derivative $u'$ is intended in the sense of distributions.
	Then there exists a linear isometric isomorphism
	\begin{equation*}
	\mathcal{J}: \mathcal{D}^{1,1} (\mathbb{R}) \to	\dot W^{1,1} (\mathbb{R}).
	\end{equation*}	
	In other words, the space $\mathcal{D}^{1,1}(\mathbb{R})$ can be identified with $\dot W^{1,1}(\mathbb{R}^n)$.	
\end{theorem}
\begin{proof}
	By basic Calculus, we know that for every $u\in C^\infty_c(\mathbb{R})$ and every $z<x<y$ we have
	\[
	|u(x) - u(y)| = \left| \int_x^y u'(s)\, d s \right| \le \int_x^y |u'(s)|\,ds,
	\]
	and
	\[
	|u(x) - u(z)| = \left| \int_z^x u'(s)\, d s \right| \le \int_z^x |u'(s)|\,ds
	\]
	By taking the limits as $y$ goes to $+\infty$ and $z$ goes to $-\infty$ in the previous inequalities, we get
	\[
	|u(x)|\le \int_x^{+\infty} |u'(s)|\,ds \qquad \mbox{ and }\qquad |u(x)|\le \int_{-\infty}^x |u'(s)|\,ds.
	\]
By summing these two estimates and passing to supremum in $x$, we have an analogue to Sobolev's inequality
	\begin{equation}
	\label{eq:p = n = 1 Sobolev}
	2\,\|u\|_{L^\infty(\mathbb{R})} \le  [u]_{W^{1,1} (\R)}, \qquad  \qquad \mbox{ for every } u \in C_c^\infty (\mathbb{R}^n).
	\end{equation}
	Therefore, every Cauchy sequence $(u_m)_{m\in\mathbb{N}}\subset C^\infty_c(\mathbb{R})$ in the $W^{1,1}$ seminorm is a Cauchy sequence in $C_0(\mathbb{R})$, as well. The latter is the Banach space of continuous functions vanishing at infinity. Thus we can infer uniform convergence of $(u_m)_{m\in\mathbb{N}}$ to a function $u\in C_0(\mathbb{R})$. Moreover, by using that $L^1(\mathbb{R})$ is a Banach space, we can infer convergence of $(u'_m)_{m\in\mathbb{N}}$ to a function $v\in L^1(\mathbb{R})$. It is easily seen that it must result
	\[
	v=u',
	\]
	thus $u\in \dot W^{1,1}(\mathbb{R})$. This argument permits to define the isometry $\mathcal{J}$, exactly as we did in the proof of Theorem \ref{thm:sp < n}. In order to prove the surjectivity of $\mathcal{J}$, it is sufficient to show that for every $u\in \dot W^{1,1}(\mathbb{R})$, there exists a sequence $(u_m)_{m\in\mathbb{N}}\subset C^\infty_c(\mathbb{R})$ such that
	\[
	\lim_{m\to\infty} [u_m-u]_{W^{1,1}(\mathbb{R})}=0.
	\]
	This is a standard fact, we leave the details to the reader.
\end{proof}
\begin{remark}
It is not difficult to see that inequality \eqref{eq:p = n = 1 Sobolev} is sharp. It is sufficient to take a sequence $(u_m)_{m\in\mathbb{N}}\subset C^\infty_c(\mathbb{R})$ such that
\[
\lim_{m\to\infty} \left([u_m-u]_{W^{1,1}(\mathbb{R})}+\|u_m-u\|_{L^\infty(\mathbb{R})}\right)=0,
\]
where $u$ is the function
\[
u(x)=\max\{1-|x|,\,0\}.
\]
Such a sequence can be constructed by standard convolution methods.
\end{remark}


\section{Comments and open questions}

\noindent \bf 1. \rm  Our three embedding theorems pose the question of what are the optimal constants, and whether they are achieved by extremal functions. We specifically refer to inequality \eqref{eq:GNS sp < n} for $s\,p<n$, inequality \eqref{PW} for $s\,p=n$, and inequality \eqref{eq:Morrey} for $s\,p>n$.

 Note that in the case $s=1$, the extremals in  the Gagliardo-Nirenberg-Sobolev range $1<p<n$ were found by Aubin \cite{Aub76} and Talenti \cite{Talenti}. The extremals in the Morrey range $p>n$, still for $s=1$, have been recently described by Hynd and Sauffert \cite{HyndSauff}. We do not know of any similar result in the limit case $p=n$.

\medskip

\noindent \bf 2. \rm In the case of proper subsets $\Omega\subset\mathbb{R}^n$ the characterization result will depend on the different options. Thus, one may take completions of $C^\infty_c(\Omega)$ with respect to:
	\begin{itemize}
	\item the full norm (but localized on $\Omega$)
\[
\|u\|_{L^p(\Omega)}+[u]_{W^{s,p}(\Omega)};	
\]
\item	the full norm (but spread all over $\mathbb{R}^n$),
\[
\|u\|_{L^p(\Omega)}+[u]_{W^{s,p}(\mathbb{R}^n)};	
\]
\item the Gagliardo-Slobodecki\u{\i} seminorm (localized on $\Omega$)
\[
[u]_{W^{s,p}(\Omega)};	
\]
\item or the Gagliardo-Slobodecki\u{\i} seminorm (spread all over $\mathbb{R}^n$).
\end{itemize}
\smallskip
In general, the resulting spaces do not coincide. See \cite[Section 2]{Brasco2019} for some comments.

\medskip

\appendix

\section{Approximation by convolution}
\label{sec:convolution}
Let $\rho\in C^\infty_c(\mathbb{R}^n)$ be a  standard Friedrichs mollifier supported on the ball $B_1(0)$, and define the sequence of smoothing kernels
\[
\rho_m(x)=m^n\,\rho(m\,x),\qquad \mbox{ for } m\in\mathbb{N}\setminus\{0\}.
\]
In \cite[Theorem 2.3]{DTGVTaylor20}, the second and third authors proved, by means of interpolation techniques, that
\begin{equation*}
\|u * \rho_m - u \|_{L^p (\mathbb R^n)} \le \frac{C}{m^s}\, [u]_{W^{s,p} (\mathbb R^n)},\qquad \mbox{ for every }u \in W^{s,p} (\mathbb R^n),\ m\ge 1,
\end{equation*}
for a constant $C=C(n)>0$. Here, we will need a stronger result (indeed, our functions need not belong to $W^{s,p}(\mathbb{R}^n)$), but without a rate of convergence. The proof is standard routine in the theory of $L^p$ spaces, we recall the argument for the reader's convenience.
\begin{lemma}
	\label{lm:cazzata}
	Let $s\in[0,1]$ and $1\le p<+\infty$. For every $u\in L^1_{\rm loc}(\mathbb{R}^n)$ such that
	\[
	[u]_{W^{s,p}(\mathbb{R}^n)}<+\infty,
	\]
	we have
	\begin{equation}
	\label{approssima}
	\lim_{m\to\infty} [u\ast \rho_m-u]_{W^{s,p}(\mathbb{R}^n)}=0.
	\end{equation}
\end{lemma}
\begin{proof}
	We focus on the case $s\in(0,1)$, the extremal cases $s=0$ and $s=1$ being simpler and well-known. By using that
	\[
	\int_{\mathbb{R}^n} \rho_m\,dx=1,\qquad \mbox{ for every } m\ge 1,
	\]
	we have
	\[
	\begin{split}
	&[u\ast \rho_m-u]_{W^{s,p}(\mathbb{R}^n)}^p\\&=s\,(1-s)\,\iint_{\mathbb{R}^n\times\mathbb{R}^n} \frac{|u\ast\rho_m(x)-u(x)-(u\ast\rho_m(y)-u(y))|^p}{|x-y|^{n+s\,p}}\,dx\,dy\\
	&=s\,(1-s)\,\iint_{\mathbb{R}^n\times\mathbb{R}^n} \frac{\left|\displaystyle\int (u(x-z)-u(x))\,\rho_m(z)\,dz-\int (u(y-z)-u(y))\,\rho_m(z)\,dz\right|^p}{|x-y|^{n+s\,p}}\,dx\,dy\\
	&=s\,(1-s)\,\iint_{\mathbb{R}^n\times\mathbb{R}^n} \frac{\left|\displaystyle\int \Big[(u(x-z)-u(x))-(u(y-z)-u(y))\Big]\,\rho_m(z)\,\,dz\right|^p}{|x-y|^{n+s\,p}}\,dx\,dy.
	\end{split}
	\]
	By using Jensen's inequality and the fact that $\rho_m(z)\,dz$ is a probability measure, we then get
	\[
	[u\ast \rho_m-u]_{W^{s,p}(\mathbb{R}^n)}^p\le s\,(1-s)\iint_{\mathbb{R}^n\times\mathbb{R}^n}\!\!\! \frac{\displaystyle\int \Big|(u(x-z)-u(x))-(u(y-z)-u(y))\Big|^p\rho_m(z)\,dz}{|x-y|^{n+s\,p}}dxdy.
	\]
	If we now exchange the order of integration, use that $\rho_m$ is supported on $B_{1/m}(0)$ and the fact that
	\[
	\lim_{|z|\to 0} \left\|H(\cdot-z,\cdot-z)-H\right\|_{L^p(\mathbb{R}^n\times\mathbb{R}^n)}=0,\qquad \mbox{ where } H(x,y)= \frac{u(x)-u(y)}{|x-y|^{\frac{n}{p}+s}} \in L^p(\mathbb{R}^n\times\mathbb{R}^n),
	\]
	we easily get the desired conclusion \eqref{approssima}.
\end{proof}

\section{Truncation lemmas}
\label{sec:6}

In what follows, we will use the shortcut notation $B_r$ for $B_r(0)$.
The aim of this section is to show that there exist smooth cut-off functions $\eta_m$ with
\begin{equation*}
	\eta_m = 1 \text{ in } B_{a_m}, \qquad \eta_m = 0 \text{ in } \mathbb{R}^n\setminus B_{b_m}, \qquad 0 \le \eta_m \le 1,
\end{equation*}
for suitable radii $a_m \le b_m$ diverging to $+\infty$,  such that  for every
$[u]_{W^{s,p} (\mathbb R^n) } < +\infty$ we have
\begin{equation*}
\lim_{m\to\infty}	[ \eta_m \, u - u ]_{W^{s,p} (\mathbb R^n) }= 0.
\end{equation*}
We will prove this for $s\,p < n$ and $s\,p = n$. For the case $s\,p > n$, we will only show that $[ \eta_m \, u - u ]_{W^{s,p} (\mathbb R^n) }$ is uniformly bounded: this is sufficient for our scope, since we can then apply weak compactness and a convexity trick based on Mazur's Lemma.

For convenience, we will write $\varphi_m = 1 - \eta_m$, and prove that
\begin{equation*}
[ \varphi_m \, u  ]_{W^{s,p} (\mathbb R^n) } \to 0 .
\end{equation*}
The main difficulty arises from the fact that only minimal integrability assumptions are assumed on $u$. In particular, {\it we do not require} that $u\in W^{s,p}(\mathbb{R}^n)$.
\subsection{Case $s\,p < n$}

For $0<s\le 1$ and $1\le p<+\infty$ such that $s\,p<n$, we recall the definition of critical Sobolev exponent
\[
p^\star_s=\frac{n\,p}{n-s\,p}.
\]
Then we have the following technical result, which is quite useful.
\begin{lemma}[Truncation lemma $s\,p<n$]
	\label{lm:truncation}
	Let
	\[
	u\in \Big\{  \varphi \in L^{p^\star_s} ( \mathbb{R}^n )\, :\,  [ \varphi ]_{W^{s,p} (\mathbb{R}^n)} < +\infty  \Big\} =:\dot W^{s,p}(\mathbb{R}^n).
	\]
	If $(\varphi_m)_{m\in\mathbb{N}}\subset C^\infty_0(\mathbb{R}^n)$ is a sequence of non-negative cut-off functions such that $0\le \varphi_m\le 1$ and
	\[
	\varphi_m\equiv 0\quad \mbox{ on } B_m,\qquad \varphi_m\equiv 1\quad \mbox{ on } \mathbb{R}^n\setminus B_{2\,m} \qquad \mbox{ and }\qquad \|\nabla\varphi_m\|_{L^\infty(\mathbb{R}^n)}\le \frac{C}{m},
	\]
	then we have
	\[
	\lim_{m\to\infty}[\varphi_m\, u]_{W^{s,p}(\mathbb{R}^n)}=0.
	\]
\end{lemma}
\begin{proof}
We deal with the case $s\in(0,1)$, the case $s=1$ being much simpler.
	We decompose the
	seminorm as follows
	\[
	\begin{split}
	\frac{1}{s\,(1-s)}\,[\varphi_m\, u]_{W^{s,p}(\mathbb{R}^n)}^p&=\iint_{\mathbb{R}^n\times\mathbb{R}^n} \frac{|\varphi_m(x)\,u(x)-\varphi_m(y)\,u(y)|^p}{|x-y|^{n+s\,p}}\,dx\,dy\\
	&=2\,\iint_{B_m\times (B_{2\,m}\setminus B_m)} \frac{|\varphi_m(y)\,u(y)|^p}{|x-y|^{n+s\,p}}\,dx\,dy\\
	&+2\,\iint_{B_m\times (\mathbb{R}^n\setminus B_{2\,m})} \frac{|u(y)|^p}{|x-y|^{n+s\,p}}\,dx\,dy\\
	&+\iint_{(B_{2\,m}\setminus B_m)\times (B_{2\,m}\setminus B_m)} \frac{|\varphi_m(x)\,u(x)-\varphi_m(y)\,u(y)|^p}{|x-y|^{n+s\,p}}\,dx\,dy\\
	&+2\,\iint_{(B_{2\,m}\setminus B_m)\times (\mathbb{R}^n\setminus B_{2\,m})} \frac{|\varphi_m(x)\,u(x)-u(y)|^p}{|x-y|^{n+s\,p}}\,dx\,dy\\
	&+\iint_{(\mathbb{R}^n\setminus B_{2\,m})\times (\mathbb{R}^n\setminus B_{2\,m})} \frac{|u(x)-u(y)|^p}{|x-y|^{n+s\,p}}\,dx\,dy\\
	&=2\,\mathcal{I}_1+2\,\mathcal{I}_2+\mathcal{I}_3+2\,\mathcal{I}_4+\mathcal{I}_5.
	\end{split}
	\]
	We show that each $\mathcal{I}_i$ converges to $0$. We treat each integral separately.
	\medskip

\noindent
	{\bf Estimate for $\mathcal{I}_1$.}
	For the first integral, we first observe that
	\[
	\begin{split}
	\mathcal{I}_1&=\iint_{B_m\times (B_{2\,m}\setminus B_m)} \frac{|\varphi_m(y)\,u(y)|^p}{|x-y|^{n+s\,p}}\,dx\,dy\\
	&=\iint_{B_m\times (B_{2\,m}\setminus B_m)} \frac{|\varphi_m(x)-\varphi_m(y)|^p\,|u(y)|^p}{|x-y|^{n+s\,p}}\,dx\,dy\\
	&\le \frac{C}{m^p}\,\iint_{B_m\times (B_{2\,m}\setminus B_m)} \frac{|u(y)|^p}{|x-y|^{n+s\,p-p}}\,dx\,dy.
	\end{split}
	\]
	By noticing that
	\[
	B_{m}\subset B_{3\,m}(y),\qquad \mbox{ for } y\in B_{2\,m}\setminus B_m,
	\]
	we get
	\[
	\int_{B_m} \frac{1}{|x-y|^{n+s\,p-p}}\,dx\le \int_{B_{3\,m}(y)}\frac{1}{|x-y|^{n+s\,p-p}}\,dx\le C\,m^{p-s\,p},
	\]
	and thus
	\[
	\mathcal{I}_1\le\frac{C}{m^{s\,p}}\,\int_{B_{2\,m}\setminus B_m} |u(y)|^p\,dy.
	\]
	The last integral can be estimated by H\"older's inequality
	\[
	\frac{1}{m^{s\,p}}\,\int_{B_{2\,m}\setminus B_m} |u(y)|^p\,dx\le C\,m^{n\,\left(1-\frac{p}{p^\star_s}\right)-s\,p}\,\left(\int_{B_{2\,m}\setminus B_m} |u(y)|^{p^\star_s}\,dx\right)^\frac{p}{p^\star_s}.
	\]
	By observing that
	\[
	n\,\left(1-\frac{p}{p^\star_s}\right)-s\,p=0,
	\]
	we conclude that $\mathcal{I}_1$ converges to $0$ as $m$ goes to $\infty$, by using that $u\in L^{p^\star_s}(\mathbb{R}^n)$ and the Dominated Convergence Theorem.
	\medskip

\noindent
	{\bf Estimate for $\mathcal{I}_2$.}  As for $\mathcal{I}_2$, by using Fubini's Theorem and H\"older's inequality, we get
	\[
	\mathcal{I}_2\le 2\,\int_{B_m}\,\left(\int_{\mathbb{R}^n\setminus B_{2\,m}} |u(y)|^{p^\star_s}\,dy\right)^\frac{n-s\,p}{n}\,\left(\int_{\mathbb{R}^n\setminus B_{2\,m}} |x-y|^{-(n+s\,p)\,\frac{n}{s\,p}}\,dy\right)^\frac{s\,p}{n}\,dx.
	\]
	Observe that
	\[
	B_m(x)\subset B_{2\,m} ,\qquad \mbox{ for every } x\in B_m,
	\]
	thus we get
	\[
	\int_{\mathbb{R}^n\setminus B_{2\,m}} |x-y|^{-(n+s\,p)\,\frac{n}{s\,p}}\,dy\le \int_{\mathbb{R}^n\setminus B_{m}(x)} |x-y|^{-(n+s\,p)\,\frac{n}{s\,p}}\,dy=C\,m^{-\frac{n^2}{s\,p}}.
	\]
	We then obtain
	\[
	\mathcal{I}_2\le C\,m^{-n}\,\left(\int_{\mathbb{R}^n\setminus B_{2\,m}} |u(y)|^{p^\star_s}\,dy\right)^\frac{n-s\,p}{n}\,\left(\int_{B_m}\,dx\right)= C\,\left(\int_{\mathbb{R}^n\setminus B_{2\,m}} |u(y)|^{p^\star_s}\,dy\right)^\frac{n-s\,p}{n}
	\]
	Since we have $u\in L^{p^\star_s}(\mathbb{R}^n)$, the last integral converges to $0$, as $m$ goes to $0$.
	\medskip

\noindent
	{\bf Estimate for $\mathcal{I}_3$.} For the third integral, we have
	\[
	\begin{split}
	\mathcal{I}_3&\le 2^{p-1}\,\iint_{(B_{2\,m}\setminus B_m)\times (B_{2\,m}\setminus B_m)} \frac{|\varphi_m(x)-\varphi_m(y)|^p}{|x-y|^{n+s\,p}}\,|u(x)|^p\,dx\,dy\\
	&+2^{p-1}\,\iint_{(B_{2\,m}\setminus B_m)\times (B_{2\,m}\setminus B_m)} \frac{|u(x)-u(y)|^p}{|x-y|^{n+s\,p}}\,|\varphi_m(y)|^p\,dx\,dy\\
	&\le \frac{C}{m^p}\,\iint_{(B_{2\,m}\setminus B_m)\times (B_{2\,m}\setminus B_m)} \frac{1}{|x-y|^{n+s\,p-p}}\,|u(x)|^p\,dx\,dy\\
	&+2^{p-1}\,\iint_{(B_{2\,m}\setminus B_m)\times (B_{2\,m}\setminus B_m)} \frac{|u(x)-u(y)|^p}{|x-y|^{n+s\,p}}\,dx\,dy.
	\end{split}
	\]
	Now, the second integral converges to $0$ by the Dominated Convergence Theorem. The first one can be handled as we did for $\mathcal{I}_1$.
	\medskip

\noindent
	{\bf Estimate for $\mathcal{I}_4$.}
	We observe that
	\[
	\begin{split}
	\mathcal{I}_4&=\iint_{(B_{2\,m}\setminus B_m)\times (\mathbb{R}^n\setminus B_{2\,m})} \frac{|\varphi_m(x)\,u(x)-u(y)|^p}{|x-y|^{n+s\,p}}\,dx\,dy\\
	&\le 2^{p-1}\,\iint_{(B_{2\,m}\setminus B_m)\times (\mathbb{R}^n\setminus B_{2\,m})} \frac{|\varphi_m(x)-\varphi_m(y)|^p\,|u(x)|^p}{|x-y|^{n+s\,p}}\,dx\,dy \\
	&+2^{p-1}\,\iint_{(B_{2\,m}\setminus B_m)\times (\mathbb{R}^n\setminus B_{2\,m})} \frac{|u(x)-u(y)|^p}{|x-y|^{n+s\,p}}\,dx\,dy,
	\end{split}
	\]
	where we used that $\varphi_m=1$ on the complement of $B_{2\,m}$. The last integral converges to $0$, while for the first one we further decompose it as follows
	\[
	\begin{split}
	\iint_{(B_{2\,m}\setminus B_m)\times (\mathbb{R}^n\setminus B_{2\,m})} &\frac{|\varphi_m(x)-\varphi_m(y)|^p\,|u(x)|^p}{|x-y|^{n+s\,p}}\,dx\,dy\\
	&=\iint_{(B_{2\,m}\setminus B_m)\times (\mathbb{R}^n\setminus B_{3\,m})} \frac{|\varphi_m(x)-\varphi_m(y)|^p\,|u(x)|^p}{|x-y|^{n+s\,p}}\,dx\,dy\\
	&+ \iint_{(B_{2\,m}\setminus B_m)\times (B_{3\,m}\setminus B_{2\,m})}\frac{|\varphi_m(x)-\varphi_m(y)|^p\,|u(x)|^p}{|x-y|^{n+s\,p}}\,dx\,dy\\
	&\le C\,\iint_{(B_{2\,m}\setminus B_m)\times (\mathbb{R}^n\setminus B_{3\,m})} \frac{|u(x)|^p}{|x-y|^{n+s\,p}}\,dx\,dy\\
	&+\frac{C}{m^p}\, \iint_{(B_{2\,m}\setminus B_m)\times (B_{3\,m}\setminus B_{2\,m})}\frac{|u(x)|^p}{|x-y|^{n+s\,p-p}}\,dx\,dy\\
	\end{split}
	\]
For the first integral, we can observe that,
if $x \in B_{2\,m}\setminus B_m$, then
\[
B_{m} (x) \subset B_{3\,m},
\]
hence $\mathbb{R}^n\setminus B_{3\,m} \subset \mathbb R^n \setminus B_{m}(x)$. This yields
\[
\begin{split}
\iint_{(B_{2\,m}\setminus B_m)\times (\mathbb{R}^n\setminus B_{3\,m})}& \frac{|u(x)|^p}{|x-y|^{n+s\,p}}\,dx\,dy\\
&= \int_{B_{2\,m}\setminus B_m}|u(x)|^p \left(\int_{\mathbb{R}^n\setminus B_{3\,m}} \frac{1}{|x-y|^{n+s\,p}}\,dy\right)\,dx\\
&\le \int_{B_{2\,m}\setminus B_m}|u(x)|^p \left(\int_{\mathbb{R}^n\setminus B_{m}(x)} \frac{1}{|x-y|^{n+s\,p}}\,dy\right)\,dx\\
&\le \frac{C}{m^{s\,p}}\,\int_{B_{2\,m}\setminus B_m}|u(x)|^p\,dx.
\end{split}
\]
The last integral converges to $0$, as we have shown while estimating $\mathcal{I}_1$. For the other remaining integral, we can use a similar estimate as for $\mathcal{I}_1$. We leave the details to the reader.
	\medskip

\noindent
	{\bf Estimate for $\mathcal{I}_5$.} This is the simplest term, we just observe that
	\[
	\lim_{n\to\infty} \mathcal{I}_5=0,
	\]
	by the Dominated Convergence Theorem. This concludes the proof.
\end{proof}

\subsection{Case $s\,p > n$}
In this case, the estimate on the truncation will be slightly worse. However, this is still sufficient for our scopes.

\begin{lemma}[Truncation lemma $s\,p>n$]
	\label{lm:truncation2}
	Let $s\in(0,1]$ and $1\le p<+\infty$ be such that $s\,p>n$.
	We set $\alpha=s-n/p$ and let
	\[
	u\in \Big\{  \varphi \in C^{0,\alpha} ( \mathbb{R}^n )\, :\,  \varphi(0)=0,\ [ \varphi ]_{W^{s,p} (\mathbb{R}^n)} < +\infty  \Big\}.
	\]
	If $(\varphi_m)_{m\in\mathbb{N}}\subset C^\infty_0(\mathbb{R}^n)$ is a sequence of non-negative cut-off functions such that $0\le \varphi_m\le 1$ and
	\[
	\varphi_m\equiv 0\quad \mbox{ on } B_m,\qquad \varphi_m\equiv 1\quad \mbox{ on } \mathbb{R}^n\setminus B_{2\,m} \qquad \mbox{ and }\qquad \|\nabla\varphi_m\|_{L^\infty(\mathbb{R}^n)}\le \frac{C}{m},
	\]
	then we have
	\[
	[\varphi_m\, u]_{W^{s,p}(\mathbb{R}^n)}\le C,
	\]
	for a constant $C>0$ not depending on $m$.
\end{lemma}
\begin{proof}
	The proof is similar to the previous one. We deal again with the case $s\in(0,1)$. We still decompose the $(s,p)-$seminorm as before and then use the estimate
	\[
	|u(x)|=|u(x)-u(0)|\le C\, |x|^{s-\frac{n}{p}},\qquad \mbox{ for every } x\in\mathbb{R}^n,
	\]
	in place of the hypothesis $L^{p^\star_s}$ previously used, in order to estimate $\mathcal{I}_1$, $\mathcal{I}_3$ and $\mathcal{I}_4$. We leave the details to the reader. For $\mathcal{I}_2$, we observe that
	\[
	\mathcal{I}_2=\iint_{B_m\times (\mathbb{R}^n\setminus B_{2\,m})} \frac{|u(y)|^p}{|x-y|^{n+s\,p}}\,dx\,dy\le C\,\iint_{B_m\times (\mathbb{R}^n\setminus B_{2\,m})} \frac{|y|^{s\,p-n}}{|x-y|^{n+s\,p}}\,dx\,dy.
	\]
	We then use that
	\[
	|x-y|\ge |y|-|x|\ge |y|-\frac{|y|}{2}=\frac{|y|}{2},\qquad \mbox{ for }x\in B_m,\, y\in \mathbb{R}^n\setminus B_{2\,m}.
	\]
	Thus we obtain
	\[
	\mathcal{I}_2\le C\,\iint_{B_m\times (\mathbb{R}^n\setminus B_{2\,m})} |y|^{-2\,n}\,dx\,dy\le C.
	\]
	This concludes the proof.
\end{proof}

\subsection{Case $s\,p = n$}
Here we can prove a result similar to Lemma \ref{lm:truncation}. For this, we will need the Poincar\'e-Wirtinger inequality of Lemma \ref{lm:gracias} and the integrability information of Lemma \ref{lm:FeSt}.
\begin{lemma}[Truncation lemma $s\,p=n$]
\label{lm:truncation3}
Let $s\in(0,1]$ and $n\ge 1$ be such that $s<n$. Let
\[
u\in \Big\{  \varphi \in\mathcal{L}^{\frac{n}{s},n} (\mathbb{R}^n )\, :\,[ \varphi ]_{W^{s,\frac{n}{s}} (\mathbb{R}^n)} < +\infty  \Big\}.
\]
If $(\varphi_m)_{m\in\mathbb{N}}\subset C^\infty_0(\mathbb{R}^n)$ is a sequence of non-negative cut-off functions such that $0\le \varphi_m\le 1$
\[
\varphi_m\equiv 0\quad \mbox{ on } B_m,\qquad \varphi_m\equiv 1\quad \mbox{ on } \mathbb{R}^n\setminus B_{m^2} \qquad \mbox{ and }\qquad \|\nabla\varphi_m\|_{L^\infty(\mathbb{R}^n)}\le \frac{C}{m^2},
\]
then we have
\[
\lim_{m\to\infty}[\varphi_m\, (u-\overline{u}_{m^2})]_{W^{s,\frac{n}{s}}(\mathbb{R}^n)}=0,\qquad \mbox{ where } \overline{u}_{m^2}= \frac{1}{|B_{m^2}|}\,\int_{B_{m^2}} u\,dx.
\]
\end{lemma}
\begin{proof}
We only consider the case $s\in(0,1)$, the local case being much simpler.
We decompose the seminorm as usual
\[
\begin{split}
\frac{1}{s\,(1-s)}\,[\varphi_m\, (u-\overline{u}_{m^2})]_{W^{s,\frac{n}{s}}(\mathbb{R}^n)}^\frac{n}{s}&=\iint_{\mathbb{R}^n\times\mathbb{R}^n} \frac{|\varphi_m(x)\,(u(x)-\overline{u}_{m^2})-\varphi_m(y)\,(u(y)-\overline{u}_{m^2})|^\frac{n}{s}}{|x-y|^{2\,n}}\,dx\,dy\\
&=2\,\iint_{B_m\times (B_{m^2}\setminus B_m)} \frac{|\varphi_m(y)\,(u(y)-\overline{u}_{m^2})|^\frac{n}{s}}{|x-y|^{2\,n}}\,dx\,dy\\
&+2\,\iint_{B_m\times (\mathbb{R}^n\setminus B_{m^2})} \frac{|u(y)-\overline{u}_{m^2}|^\frac{n}{s}}{|x-y|^{2\,n}}\,dx\,dy\\
&+\iint_{(B_{m^2}\setminus B_m)\times (B_{m^2}\setminus B_m)} \frac{|\varphi_m(x)\,(u(x)-\overline{u}_{m^2})-\varphi_m(y)\,(u(y)-\overline{u}_{m^2})|^\frac{n}{s}}{|x-y|^{2\,n}}\,dx\,dy\\
&+2\,\iint_{(B_{m^2}\setminus B_m)\times (\mathbb{R}^n\setminus B_{m^2})} \frac{|\varphi_m(x)\,(u(x)-\overline{u}_{m^2})-(u(y)-\overline{u}_{m^2})|^\frac{n}{s}}{|x-y|^{2\,n}}\,dx\,dy\\
&+\iint_{(\mathbb{R}^n\setminus B_{m^2})\times (\mathbb{R}^n\setminus B_{m^2})} \frac{|u(x)-u(y)|^\frac{n}{s}}{|x-y|^{2\,n}}\,dx\,dy\\
&=2\,\mathcal{I}_1+2\,\mathcal{I}_2+\mathcal{I}_3+2\,\mathcal{I}_4+\mathcal{I}_5.
\end{split}
\]
 We treat each integral separately.
\vskip.2cm\noindent
{\bf Estimate for $\mathcal{I}_1$.}
For the term $\mathcal{I}_1$, we first observe that for every $y\in B_{m^2}\setminus B_m$ and every $x\in B_m$, we have
\[
\begin{split}
|\varphi_m(y)|=|\varphi_m(y)-\varphi_m(x)|&\le \frac{C}{m^2}\,|y-x|.
\end{split}
\]
This entails
\[
\begin{split}
\mathcal{I}_1&=\iint_{B_m\times (B_{m^2}\setminus B_m)} \frac{|\varphi_m(y)\,(u(y)-\overline{u}_{m^2})|^\frac{n}{s}}{|x-y|^{2\,n}}\,dx\,dy\\
&\le \frac{C}{(m^2)^\frac{n}{s}}\,\iint_{B_m\times (B_{m^2}\setminus B_m)} \frac{|u(y)-\overline{u}_{m^2}|^\frac{n}{s}}{|x-y|^{2\,n-\frac{n}{s}}}\,dx\,dy.
\end{split}
\]
By noticing that
\[
B_{m}\subset B_{m^2+m}(y),\qquad \mbox{ for } y\in B_{m^2}\setminus B_m,
\]
we get
\[
\int_{B_m} \frac{1}{|x-y|^{2\,n-\frac{n}{s}}}\,dx\le \int_{B_{m^2+m}(y)}\frac{1}{|x-y|^{2\,n-\frac{n}{s}}}\,dx\le C\,(m^2)^{\frac{n}{s}-n},
\]
and thus
\[
\mathcal{I}_1\le \frac{C}{m^{2\,n}}\,\int_{B_{m^2}\setminus B_m} |u(y)-\overline{u}_{m^2}|^\frac{n}{s}\,dy.
\]
If we now apply Lemma \ref{lm:gracias} on the right-hand side, we get
\[
\mathcal{I}_1\le C\,\iint_{(B_{m^2}\setminus B_m)\times B_{m^2}} \frac{|u(x)-u(y)|^\frac{n}{s}}{|x-y|^{2\,n}}\,dx\,dy.
\]
By using that $[u]_{W^{s,n/s}(\mathbb{R}^n)}<+\infty$ and the Dominated Convergence Theorem, we get that the last term above converges to $0$, as $m$ goes to $\infty$.
\medskip

\noindent
{\bf Estimate for $\mathcal{I}_2$.} This term now is quite delicate, here we need a global integrability information on $u$.
We observe that for every $m\ge 2$, $x\in B_m$ and $y\in\mathbb{R}^n\setminus B_{m^2}$, we have
\[
|x-y|\ge |y|-|x|\ge |y|-\sqrt{|y|}\ge \frac{|y|}{2}.
\]
Thus we get
\begin{equation}
\label{addedafter}
\begin{split}
\mathcal{I}_2&\le C\,m^n\, \int_{\mathbb{R}^n\setminus B_{m^2}}\frac{|u(y)-\overline{u}_{m^2}|^\frac{n}{s}}{|y|^{2\,n}}\,dy\\
&\le C\,\int_{\mathbb{R}^n\setminus B_{m^2}} |y|^\frac{n}{2}\,\frac{|u(y)-\overline{u}_{m^2}|^\frac{n}{s}}{|y|^{2\,n}}\,dy= \int_{\mathbb{R}^n\setminus B_{m^2}}\frac{|u(y)-\overline{u}_{m^2}|^\frac{n}{s}}{|y|^{\frac{3\,n}{2}}}\,dy.
\end{split}
\end{equation}
We now observe that for $m\ge 2$, $|y|\ge m^2$ and $0<\alpha<n/2$, we have
\[
|y|^\frac{3\,n}{2}\ge |y|^\frac{n}{2}\,m^{2\,n}\ge |y|^\alpha\,m^{2\,n},
\]
and also
\[
\begin{split}
|y|^\frac{3\,n}{2}&=|y|^n\,\left(\log\left(\frac{|y|}{m^2}\right)\right)^{\frac{n}{s}+2}\,|y|^\frac{n}{2}\,\left(\log\left(\frac{|y|}{m^2}\right)\right)^{-\frac{n}{s}-2}\\
&\ge |y|^n\,\left(\log\left(\frac{|y|}{m^2}\right)\right)^{\frac{n}{s}+2}\,\frac{1}{C_\alpha}\,|y|^\alpha,
\end{split}
\]
for a suitable constant $C_\alpha>1$. By combining the two previous estimates, we then obtain
\[
\begin{split}
|y|^\frac{3\,n}{2}&\ge \frac{|y|^\alpha}{2}\,\left(m^{2\,n}+\frac{1}{C_\alpha}\,|y|^n\,\left(\log\left(\frac{|y|}{m^2}\right)\right)^{\frac{n}{s}+2}\right)\\
&\ge \frac{|y|^\alpha}{2\,C_\alpha}\,\left(m^{2\,n}+|y|^n\,\left(\log\left(\frac{|y|}{m^2}\right)\right)^{\frac{n}{s}+2}\right).
\end{split}
\]
We can then estimate the integral in the right-hand side of \eqref{addedafter} as follows
\[
\begin{split}
\int_{\mathbb{R}^n\setminus B_{m^2}}\frac{|u(y)-\overline{u}_{m^2}|^\frac{n}{s}}{|y|^{\frac{3\,n}{2}}}\,dy&\le 2\,C_\alpha\,\int_{\mathbb{R}^n\setminus B_{m^2}}\frac{|u(y)-\overline{u}_{m^2}|^\frac{n}{s}}{m^{2\,n}+|y|^n\,\left(\log\dfrac{|y|}{m^2}\right)^{\frac{n}{s}+2}}\, \frac{dy}{|y|^\alpha}\\
&\le \frac{2\,C_\alpha}{m^{2\,\alpha}}\,\int_{\mathbb{R}^n}\frac{|u(y)-\overline{u}_{m^2}|^\frac{n}{s}}{m^{2\,n}+|y|^n\,\left|\log\dfrac{|y|}{m^2}\right|^{\frac{n}{s}+2}}\, dy.
\end{split}
\]
By using Lemma \ref{lm:FeSt} with $p=n/s$,
we get that $\mathcal{I}_2$ converges to $0$, as $m$ goes to $\infty$.
\medskip

\noindent
{\bf Estimate for $\mathcal{I}_3$.} For the third integral, we have
\[
\begin{split}
\mathcal{I}_3&\le 2^{\frac{n}{s}-1}\,\iint_{(B_{m^2}\setminus B_m)\times (B_{m^2}\setminus B_m)} \frac{|\varphi_m(x)-\varphi_m(y)|^\frac{n}{s}}{|x-y|^{2\,n}}\,|u(x)-\overline{u}_{m^2}|^\frac{n}{s}\,dx\,dy\\
&+2^{\frac{n}{s}-1}\,\iint_{(B_{m^2}\setminus B_m)\times (B_{m^2}\setminus B_m)} \frac{|u(x)-u(y)|^\frac{n}{s}}{|x-y|^{2\,n}}\,|\varphi_m(y)|^\frac{n}{s}\,dx\,dy\\
&\le \frac{C}{(m^2)^\frac{n}{s}}\,\iint_{(B_{m^2}\setminus B_m)\times (B_{m^2}\setminus B_m)} \frac{1}{|x-y|^{2\,n-\frac{n}{s}}}\,|u(x)-\overline{u}_{m^2}|^\frac{n}{s}\,dx\,dy\\
&+2^{\frac{n}{s}-1}\,\iint_{(B_{m^2}\setminus B_m)\times (B_{m^2}\setminus B_m)} \frac{|u(x)-u(y)|^\frac{n}{s}}{|x-y|^{2\,n}}\,dx\,dy.
\end{split}
\]
Now, the second integral converges to $0$ by the Dominated Convergence Theorem. The first one can be handled as we did for $\mathcal{I}_1$.
\medskip

\noindent
{\bf Estimate for $\mathcal{I}_4$.}
Here, we observe that
\[
\begin{split}
\mathcal{I}_4&=\iint_{(B_{m^2}\setminus B_m)\times (\mathbb{R}^n\setminus B_{m^2})} \frac{|\varphi_m(x)\,(u(x)-\overline{u}_{m^2})-(u(y)-\overline{u}_{m^2})|^\frac{n}{s}}{|x-y|^{2\,n}}\,dx\,dy\\
&\le 2^{\frac{n}{s}-1}\,\iint_{(B_{m^2}\setminus B_m)\times (\mathbb{R}^n\setminus B_{m^2})} \frac{|\varphi_m(x)-\varphi_m(y)|^\frac{n}{s}\,|u(x)-\overline{u}_{m^2}|^\frac{n}{s}}{|x-y|^{2\,n}}\,dx\,dy \\
&+2^{\frac{n}{s}-1}\,\iint_{(B_{m^2}\setminus B_m)\times (\mathbb{R}^n\setminus B_{m^2})} \frac{|u(x)-u(y)|^\frac{n}{s}}{|x-y|^{2\,n}}\,dx\,dy,
\end{split}
\]
where we used that $\varphi_m=1$ on the complement of $B_{m^2}$. The last integral converges to $0$, while the first one can be decomposed as follows:
\[
\begin{split}
\iint_{(B_{m^2}\setminus B_m)\times (\mathbb{R}^n\setminus B_{m^2})} &\frac{|\varphi_m(x)-\varphi_m(y)|^\frac{n}{s} \,|u(x)-\overline{u}_{m^2}|^\frac{n}{s}}{|x-y|^{2\,n}}\,dx\,dy\\
&=\iint_{(B_{m^2}\setminus B_m)\times (\mathbb{R}^n\setminus B_{2\,m^2})} \frac{|\varphi_m(x)-\varphi_m(y)|^\frac{n}{s} \,|u(x)-\overline{u}_{m^2}|^\frac{n}{s}}{|x-y|^{2\,n}}\,dx\,dy\\
&+ \iint_{(B_{m^2}\setminus B_m)\times (B_{2\,m^2}\setminus B_{m^2})}\frac{|\varphi_m(x)-\varphi_m(y)|^\frac{n}{s} \,|u(x)-\overline{u}_{m^2}|^\frac{n}{s}}{|x-y|^{2\,n}}\,dx\,dy\\
&\le C\,\iint_{(B_{m^2}\setminus B_m)\times (\mathbb{R}^n\setminus B_{2\,m^2})} \frac{|u(x)-\overline{u}_{m^2}|^\frac{n}{s}}{|x-y|^{2\,n}}\,dx\,dy\\
&+\frac{C}{(m^2)^\frac{n}{s}}\, \iint_{(B_{m^2}\setminus B_m)\times (B_{2\,m^2}\setminus B_{m^2})}\frac{|u(x)-\overline{u}_{m^2}|^\frac{n}{s}}{|x-y|^{2\,n-\frac{n}{s}}}\,dx\,dy.
\end{split}
\]
For the first integral, we can observe that,
if $x \in B_{m^2}\setminus B_m$, then
\[
B_{m^2} (x) \subset B_{2\,m^2},
\]
hence $\mathbb{R}^n\setminus B_{2\,m^2} \subset \mathbb R^n \setminus B_{m^2}(x)$. This yields
\[
\begin{split}
\iint_{(B_{m^2}\setminus B_m)\times (\mathbb{R}^n\setminus B_{2\,m^2})}& \frac{|u(x)-\overline{u}_{m^2}|^\frac{n}{s}}{|x-y|^{2\,n}}\,dx\,dy\\
&= \int_{B_{m^2}\setminus B_m}|u(x)-\overline{u}_{m^2}|^\frac{n}{s} \left(\int_{\mathbb{R}^n\setminus B_{2\,m^2}} \frac{1}{|x-y|^{2\,n}}\,dy\right)\,dx\\
&\le \int_{B_{m^2}\setminus B_m}|u(x)-\overline{u}_{m^2}|^\frac{n}{s} \left(\int_{\mathbb{R}^n\setminus B_{m^2}(x)} \frac{1}{|x-y|^{2\,n}}\,dy\right)\,dx\\
&\le \frac{C}{(m^2)^n}\,\int_{B_{m^2}\setminus B_m}|u(x)-\overline{u}_{m^2}|^\frac{n}{s}\,dx.
\end{split}
\]
Proceeding as for $\mathcal{I}_1$, the last term converges to $0$.
\par
For the second integral we use a similar estimate as for $\mathcal{I}_1$. More precisely, we observe that for $x\in B_{m^2}$, we have
$
B_{2\,m^2}\setminus B_{m^2}\subset B_{3\,m^2}(x).
$
Thus we can estimate
\[
\begin{split}
\frac{1}{(m^2)^\frac{n}{s}} \iint_{(B_{m^2}\setminus B_m)\times (B_{2\,m^2}\setminus B_{m^2})}
	&\frac{|u(x)-\overline{u}_{m^2}|^\frac{n}{s}}{|x-y|^{2\,n-\frac{n}{s}}}\,dx\,dy \\
	&= \frac{1}{(m^2)^\frac{n}{s}}\, \int_{B_{m^2}\setminus B_m}|u(x)-\overline{u}_{m^2}|^\frac{n}{s} \left(\int_{B_{2\,m^2}\setminus B_{m^2}} \frac{1}{|x-y|^{2\,n - \frac n s}}\,dy\right)\,dx \\
	& \le \frac{1}{(m^2)^\frac{n}{s}}\, \int_{B_{m^2}\setminus B_m}|u(x)-\overline{u}_{m^2}|^\frac{n}{s} \left(\int_{B_{3\,m^2 (x)}} \frac{1}{|x-y|^{2\,n - \frac n s}}\,dy\right)\,dx \\
	& =\frac{C}{(m^2)^n}\,\int_{B_{m^2}\setminus B_m}|u(x)-\overline{u}_{m^2}|^\frac{n}{s} d x.
\end{split}
\]
The latter is again the same term previously treated.
\vskip.2cm\noindent
{\bf Estimate for $\mathcal{I}_5$.} As usual, this is the simplest term, it results
\[
\lim_{n\to\infty} \mathcal{I}_5=0,
\]
by the Dominated Convergence Theorem. The desired conclusion now follows.
\end{proof}

\section*{Acknowledgments}
An anonymous referee is gratefully acknowledged for the careful reading of the manuscript and for the many useful comments. We also wish to thank Alessandro Monguzzi for having kindly drawn our attention to the paper \cite{MPS}.
L. Brasco was financially supported by the grant FFABR {\it Fondo Per il Finanziamento delle attivit\`a di base}, of the Italian Government.
The work of D. G\'omez-Castro  and J. L. V\'azquez were funded by grant PGC2018-098440-B-I00 from  the  MICINN,  of the Spanish Government.
The research of D. G\'omez-Castro was supported by the Advanced Grant {\it Nonlocal -- CPD (Nonlocal PDEs for Complex Particle Dynamics:
Phase Transitions, Patterns and Synchronization)} of the European Research Council Executive Agency (ERC) under the European Union's Horizon 2020 research and innovation program (grant agreement No.\ 883363).
J.~L.~V\'azquez is an Honorary Professor at Universidad Complutense de Madrid.

\end{document}